\documentclass{svmult}

\usepackage{ifxetex,ifluatex}
\newif\ifxetexorluatex
\ifxetex
  \xetexorluatextrue
\else
  \ifluatex
    \xetexorluatextrue
  \else
    \xetexorluatexfalse
  \fi
\fi

\usepackage{amsmath}
\ifxetexorluatex
  \usepackage{fontspec}           
  \usepackage{unicode-math}       
  \usepackage{polyglossia}        
  \setmainlanguage{english}

  \AtBeginDocument{\let\setminus\smallsetminus}
\else
  \usepackage[utf8]{inputenc}     
  \usepackage[english]{babel}     

  \usepackage{mathptmx}       
  \usepackage{helvet}         
  \usepackage{courier}        

  \usepackage{amsfonts}
  \usepackage{amssymb}
  \newcommand{\nsimeq}{\not\simeq}
\fi

\usepackage{xcolor}          
\usepackage[bottom]{footmisc}
\usepackage{stmaryrd}        
\usepackage{mathtools}
\usepackage{enumitem}

\definecolor{linkred}{RGB}{255,1,1}
\definecolor{citegreen}{RGB}{1,190,1}
\usepackage[linkcolor=linkred
           ,citecolor=citegreen
           ,ocgcolorlinks        
           ,bookmarksopen=True,  
           ]{hyperref}

\usepackage{microtype}           

\makeatletter
\def\toclevel@authorch{1000}
\def\toclevel@author{1000}
\makeatother

\newenvironment{equivenumerate}[1][]
{\enumerate[label=\textup(\alph*\textup),ref=(\alph*),#1]}
{\endenumerate}
\setlist[enumerate,1]{label=\textup{(\arabic*)}, ref=(\arabic*), leftmargin=0.75cm}
\setlist[enumerate,2]{label=(\roman*), ref=(\roman*)}

\makeatletter

\newcommand{\sc@lettershortcut}[3]{%
  \expandafter\providecommand\csname #2#3\endcsname{#1{#3}}%
}

\newcommand{\sc@shortcuts}[3]{%
  \count@=0
  \loop
  \advance\count@ 1
  \edef\tmp@{%
    \noexpand\sc@lettershortcut\unexpanded{{#1}}{#2}{#3\count@}
  }
  \tmp@
  \ifnum\count@<26
  \repeat
}

\newcommand{\defshortcuts}[2]{\sc@shortcuts{#1}{#2}{\@alph}}

\newcommand{\defShortcuts}[2]{\sc@shortcuts{#1}{#2}{\@Alph}}

\makeatother
\let\sc\undefined
\defShortcuts{\mathbb}{b}
\defShortcuts{\mathcal}{c}
\defShortcuts{\mathfrak}{f}
\defShortcuts{\mathsf}{s}
\defshortcuts{\mathfrak}{f}
\defshortcuts{\mathsf}{s}

\def\rfop{*}
\makeatletter
\newcommand\rigidfactorization[2][]{%
  \def\rf@delim{\rfop}
  \newif\ifrf@notfirst
  #1
  \@for\next:=#2\do{%
    \ifrf@notfirst
      \rf@delim
    \fi
    \rf@notfirsttrue
    \next
  }%
}
\makeatother
\newcommand\rf\rigidfactorization
\DeclarePairedDelimiter{\length}{\lvert}{\rvert}
\DeclarePairedDelimiter{\card}{\lvert}{\rvert}
\DeclarePairedDelimiter{\babs}{\big\lvert}{\big\rvert}
\DeclareMathOperator{\ann}{ann}
\DeclareMathOperator{\udim}{udim}
\DeclareMathOperator{\Hom}{Hom}
\DeclareMathOperator{\id}{id}
\DeclareMathOperator{\gr}{gr}
\DeclareMathOperator{\tr}{\,tr\,}
\DeclareMathOperator{\nr}{nr}
\DeclareMathOperator{\chr}{char}
\DeclareMathOperator{\Cl}{\mathcal C}
\DeclareMathOperator{\LF}{LF}
\DeclareMathOperator{\gldim}{gldim}
\DeclareMathOperator{\adj}{adj}
\DeclareMathOperator{\Ob}{Ob}
\newcommand{\iv}[2]{[{#1},{#2}]}
\newcommand{\dsim}{\textup{sim}}  

\spnewtheorem{theorem}{Theorem}[section]{\bfseries}{\itshape}

\spnewtheorem{lemma}[theorem]{Lemma}{\bfseries}{\itshape}
\spnewtheorem{definition}[theorem]{Definition}{\bfseries}{\upshape}
\spnewtheorem{proposition}[theorem]{Proposition}{\bfseries}{\itshape}
\spnewtheorem{corollary}[theorem]{Corollary}{\bfseries}{\itshape}
\spnewtheorem*{remark}{Remark}{\itshape}{\rmfamily}
\spnewtheorem{example}[theorem]{Example}{\bfseries}{\upshape}

\begin{document}

\title*{Factorizations of Elements in Noncommutative Rings: A survey}
\author{Daniel Smertnig}
\institute{Daniel Smertnig \at University of Graz,
NAWI Graz,
Institute for Mathematics and Scientific Computing,
Heinrichstraße 36,
8010 Graz, Austria, \email{daniel.smertnig@uni-graz.at}}
%
%
\maketitle

\renewcommand*{\theHtheorem}{\theHsection.\arabic{theorem}}

\emph{Dedicated to Franz Halter-Koch on the occasion of his 70\textsuperscript{th} birthday.}
\vspace{1cm}

\abstract{
  We survey results on factorizations of non-zero-divisors into atoms (irreducible elements) in noncommutative rings.
  The point of view in this survey is motivated by the commutative theory of non-unique factorizations.
  Topics covered include unique factorization up to order and similarity, $2$-firs, and modular LCM domains, as well as UFRs and UFDs in the sense of Chatters and Jordan and generalizations thereof.
  We recall arithmetical invariants for the study of non-unique factorizations, and give transfer results for arithmetical invariants in matrix rings, rings of triangular matrices, and classical maximal orders as well as classical hereditary orders in central simple algebras over global fields.
}

\section{Introduction}

Factorizations of elements in a ring into atoms (irreducible elements) are natural objects to study if one wants to understand the arithmetic of a ring.
In this overview, we focus on the semigroup of non-zero-divisors in noncommutative (associative, unital) rings.
The point of view in this article is motivated by analogy with the commutative theory of non-unique factorizations (as in \cite{anderson97, chapman05, ghk06, geroldinger09}).

We start by giving a rigorous notion of \emph{rigid factorizations} and discussing sufficient conditions for the existence of factorizations of any non-zero-divisor, in Section~\ref{sec:rf}.
In Section~\ref{sec:unique-factorization}, we look at several notions of \emph{factoriality}, that is, notions of unique factorization, that have been introduced in the noncommutative setting.
Finally, in Section~\ref{sec:non-unique-factorization} we shift our attention to non-unique factorizations and the study of arithmetical invariants used to describe them.

The investigation of factorizations in noncommutative rings has its origins in the study of homogeneous linear differential equations.
The first results on the uniqueness of factorizations of linear differential operators are due to Landau, in \cite{landau02}, and Loewy, in \cite{loewy03}.
Ore, in \cite{ore33}, put this into an entirely algebraic context by studying skew polynomials (also called Ore extensions) over division rings.
He showed that if $D$ is a division ring, then the skew polynomial ring $D[x;\sigma,\delta]$, where $\sigma$ is an injective endomorphism of $D$ and $\delta$ is a $\sigma$-derivation, satisfies an Euclidean algorithm with respect to the degree function.
Hence, factorizations of elements in $D[x;\sigma,\delta]$ are unique up to order and \emph{similarity}.
We say that $D[x;\sigma,\delta]$ is \emph{similarity factorial} (see Definition~\ref{def:similarity-factorial}).

Jacobson, in \cite{jacobson43}, already describes unique factorization properties for principal ideal domains.
He showed that PIDs are similarity factorial.
In a further generalization, principal ideal domains were replaced by \emph{2-firs}, and the Euclidean algorithm was replaced by the \emph{$2$-term weak algorithm}.
This goes back to work primarily due to P.\,M.~Cohn and Bergman.
 The main reference is \cite{cohn06}.

Factorizations in $2$-firs, the $2$-term weak algorithm, and the notion of \emph{similarity factoriality} are the focus of Section~\ref{ssec:similarity-ufds}.
A key result is that the free associative algebra $K\langle X \rangle$ over a field $K$ in a family of indeterminates $X$ is similarity factorial.
Here, $K$ cannot be replaced by an arbitrary factorial domain, as $\bZ\langle x,y\rangle$ is not similarity factorial.
Brungs, in \cite{brungs69}, studied the slightly weaker notion of \emph{subsimilarity factoriality}.
Using a form of Nagata's Theorem, it follows that free associative algebras over factorial commutative domains are subsimilarity factorial.

Modular right LCM domains were studied by Beauregard in a series of papers and are also discussed in Section~\ref{ssec:similarity-ufds}.
Many results on unique factorizations in Section~\ref{ssec:similarity-ufds} can be derived from the Jordan-Hölder Theorem on (semi-)modular lattices by consideration of a suitable lattice.
Previous surveys covering unique factorizations in noncommutative rings, as considered in Section~\ref{ssec:similarity-ufds}, are \cite{cohn63,cohn65} and \cite{cohn73,cohn73-c}.
We also refer to the two books \cite{cohn85} and \cite{cohn06}.

A rather different notion of [Noetherian] UFRs (unique factorization rings) and UFDs (unique factorization domains), originally introduced by Chatters and Jordan in \cite{chatters84,chatters-jordan86}, has seen widespread adoption in ring theory.
We discuss this concept, and its generalizations, in Section~\ref{ssec:chatters-ufds}.
Examples of Noetherian UFDs include universal enveloping algebras of finite-dimensional solvable Lie algebras over $\bC$, various (semi)group algebras, and quantum algebras.
In a UFR $R$, the semigroup of nonzero normal elements, $N(R)^\bullet$, is a UF-monoid.
Thus, nonzero normal elements of $R$ factor uniquely as products of prime elements.

Section~\ref{sec:non-unique-factorization} is devoted to the study of non-unique factorizations in noncommutative rings.
Here, the basic interest is in determining arithmetical invariants that suitably measure, characterize, or describe the extent of non-uniqueness of the factorizations.
A recent result by Bell, Heinle, and Levandovskyy, from \cite{bell-heinle-levandovskyy}, establishes that many interesting classes of noncommutative domains are finite factorization domains (FF-domains).

We recall several arithmetical invariants, as well as the notion of \emph{[weak] transfer homomorphisms}.
Transfer homomorphisms have played a central role in the commutative theory of non-unique factorizations and promise to be useful in the noncommutative setting as well.
By means of transfer results, it is sometimes possible to reduce the study of arithmetical invariants in a ring to the study of arithmetical invariants in a much simpler object.

Most useful are transfer results from the non-zero-divisors of a noncommutative ring to a commutative ring or semigroup for which the factorization theory is well understood.
Such transfer results exist for rings of triangular matrices (see \cite{bachman-baeth-gossell14,baeth-smertnig15}), rings of matrices (see \cite{estes-matijevic79a,estes-matijevic79b}), and classical hereditary (in particular, maximal) orders in central simple algebras over global fields (see \cite{estes-nipp89,estes91a,estes91b,smertnig13,baeth-smertnig15}).
These results are covered in Section~\ref{ssec:transfer-results}.

Throughout the text, we gather known examples from the literature and point out their implications for factorization theory.
In particular, these examples demonstrate limitations of certain concepts or methods in the noncommutative setting when compared to the commutative setting.

As a note on terminology, we  call a domain \emph{similarity [subsimilarity,projectivity] factorial} instead of a \emph{similarity-[subsimilarity,projectivity]-UFD}.
This matches the terminology presently preferred in the commutative setting.
Using an adjective to describe the property sometimes makes it easier to use it in writing.
Moreover, this allows us to visibly differentiate factorial domains from the [Noetherian] UFRs and UFDs in the sense of Chatters and Jordan that are discussed in Section~\ref{ssec:chatters-ufds}.

While an attempt has been made to be comprehensive, it would be excessive to claim the results contained in this article are entirely exhaustive.
Many interesting results on non-unique factorizations are scattered throughout the literature, with seemingly little previous effort to tie them together under a common umbrella of a theory of (non-unique) factorizations.

Naturally, there are certain restrictions on the scope of the present treatment.
For the reader who came expecting something else under the heading \emph{factorization theory}, some pointers to recent work, which is beyond the scope of this article, but may conceivably be considered to be factorization theory, are given in Section~\ref{sec:scope}.

\section{Preliminaries}

All rings are assumed to be unital and associative, but not necessarily commutative.
All semigroups have a neutral element.
A ring $R$ is a \emph{domain} if $0$ is the unique zero-divisor (in particular, $R \ne 0$).
A \emph{right principal ideal domain (right PID)} is a domain in which every right ideal is principal.
A \emph{left PID} is defined analogously, and a domain is a \emph{principal ideal domain (PID)} if it is both, a left and a right PID.
We make similar conventions for other notions for which a left and a right variant exist, e.g. Noetherian, Euclidean, etc.

\subsection{Small Categories as Generalizations of Semigroups}

We will be interested in factorizations of non-zero-divisors in a ring $R$.
Even so, it will sometimes be useful to have the notions of factorizations available in the more general setting of semigroups, or even more generally, in the setting of small categories.
Thus, we develop the basic terminology in the very general setting of a cancellative small category.
This generality does not cause any significant additional problems over making the definitions in a more restrictive setting, such as cancellative semigroups, or even the semigroup of non-zero-divisors in a ring.
It may however be useful to keep in mind that the most important case for us will be where the cancellative small category simply is the semigroup of non-zero-divisors of a ring.

Here, a small category is viewed as a generalization of a semigroup, in the sense that the category of semigroups is equivalent to the category of small categories with a single object.
In practice, we will however be concerned mostly with semigroups.
Therefore, we use a notation for small categories that is reminiscent of that for semigroups.
We briefly review the notation.
See also \cite[Section 2.1]{smertnig13} and \cite[Section 2]{baeth-smertnig15} for more details.

Let $H$ be a small category.
A morphism $a$ of $H$ has a source $s(a)$ and a target $t(a)$.
If $a$ and $b$ are morphisms with $t(a)=s(b)$ we write the composition left to right as $ab$.
The objects of the category will play no significant role (they can always be recovered from the morphisms via the source and target maps).
We identify the objects with their identity morphisms and denote the set of all identity morphism by $H_0$.
We identify $H$ with its set of morphisms.
Accordingly, we call a morphism $a$ of $H$ simply an element of $H$ and write $a \in H$.

More formally, from this point of view, a small category $H=(H,H_0,s,t,\cdot)$ consists of the following data: A set $H$ together with a distinguished subset $H_0 \subset H$, two functions $s$,~$t\colon H \to H_0$ and a partial function $\cdot\colon H \times H \to H$ such that:
\begin{enumerate}
  \item $s(e)=t(e)=e$ for all $e \in H_0$,
  \item $a \cdot b \in H$ is defined for all $a$,~$b \in H$ with $t(a)=s(b)$,
  \item $a\cdot(b\cdot c) = (a \cdot b)\cdot c$ for all $a$,~$b$,~$c \in H$ with $t(a)=s(b)$ and $t(b)=s(c)$,
  \item $s(a) \cdot a = a \cdot t(a) = a$ for all $a \in H$.
\end{enumerate}

For $e$,~$f \in H_0$, we define $H(e,\cdot) = \{\, a \in H \mid s(a)=e \,\}$, $H(\cdot,f) = \{\, a \in H \mid t(a) = f \,\}$, $H(e,f) = H(e,\cdot) \cap H(\cdot,f)$, and $H(e) = H(e,e)$.

To see the equivalence of this definition with the usual definition of a small category, suppose first that $H$ is as above.
Take as set of objects of a category $\cC$ the set $H_0$, and, for two objects $e$, $f \in H_0$, set $\Hom_\cC(e,f) = H(f,e)$.
Define the composition on $\cC$ using the partial map $\cdot$.
Then $\cC$ is a small category in terms of the usual definition, with composition written right to left and with $e \in \Hom(e,e)$ the identity morphism of the object $e$.
Conversely, if $\cC$ is a small category in the usual sense, set $H = \bigcup_{e,f \in \Ob \cC} \Hom_{\cC}(e,f)$ and $H_0 = \{\, \id_e \mid e \in \Ob \cC \,\}$.
For $a \in H$ with domain $e$ and codomain $f$, set $s(a)=f$ and $t(a)=e$.
The partial function $\cdot$ on $H$ is defined via the composition of $\cC$.
Then $H$ satisfies the properties above.

Let $H$ be a small category.
If $a$,~$b \in H$ and we write $ab$, we implicitly assume $t(a)=s(b)$.
The subcategory of units (isomorphisms) of $H$ is denoted by $H^\times$.
The small category $H$ is a \emph{groupoid} if $H=H^\times$, and it is \emph{reduced} if $H^\times=H_0$.
An element $a \in H$ is \emph{cancellative} if it is both a monomorphism and an epimorphism, that is, for all $b$, $c$ in $H$, $ab=ac$ implies $b=c$ and $ba=ca$ implies $b=c$.
The subcategory of cancellative elements of $H$ is denoted by $H^\bullet$.
A functor $f$ from $H$ to another small category $H'$ is referred to as a \emph{homomorphism}.
Two elements $a$,~$b \in H$ are (two-sided) \emph{associated} if there exist $\varepsilon$,~$\eta \in H^\times$ such that $a=\varepsilon b \eta$.

Let $H$ be a small category.
A subset $I \subset H$ is a right ideal of $H$ if $IH=\{\, xa \mid x \in I,\, a \in H: t(x)=s(a) \,\}$ is a subset of $I$.
A right ideal of $H$ is called a \emph{right $H$-ideal} if there exists an $a \in H^\bullet$ such that $a \in I$.
A right ideal $I \subset H$ is \emph{principal} if there exists $a \in H$ such that $I=aH$.
An ideal $I \subset H$ is \emph{principal} if it is principal as a left and right ideal, that is, there exist $a$,~$b \in H$ such that $I=aH=Hb$.
Suppose that every left or right divisor of a cancellative element is again cancellative.
If $I \subset H$ is an ideal and $I=Ha=bH$ with $a$,~$b \in H^\bullet$, then it is easy to check that also $I=aH=Hb$.

Let $H$ be a semigroup.
An element $a \in H$ is \emph{normal} (or \emph{invariant}) if $aH=Ha$.
We write $N(H)$ for the subsemigroup of all normal elements of $H$.
The semigroup $H$ is \emph{normalizing} if $H=N(H)$.

In the commutative theory of non-unique factorizations, a \emph{monoid} is usually defined to be a cancellative commutative semigroup.
Since the meaning of \emph{monoid} in articles dealing with a noncommutative setting is often different, we will avoid its use altogether.
The exception are compound nouns such as \emph{Krull monoid}, \emph{free monoid}, \emph{free abelian monoid}, \emph{monoid of zero-sum sequences}, and \emph{UF-monoid}, where the use of \emph{monoid} is universal and it would be strange to introduce different terminology.

\subsection{Classical Maximal Orders}

Classical maximal orders in central simple algebras over a global field will appear throughout in examples.
Moreover, they are one of the main objects for which we are interested in studying non-unique factorizations.
Therefore, we recall the setting.
We use \cite{reiner75} as a general reference, and \cite{curtis-reiner87,swan80} for strong approximation.
For the motivation for calling such orders \emph{classical} orders, and the connection to different notions of orders, see \cite[\S5.3]{mcconnell-robson01}.

Let $K$ be a global field, that is, either an algebraic number field or an algebraic function field (of transcendence degree 1) over a finite field.
Let $S_{\text{fin}}$ denote the set of all non-archimedean places of $K$.
For each $v \in S_{\text{fin}}$, let $\cO_v \subset K$ denote the corresponding discrete valuation domain.
A subring $\cO \subset K$ is a \emph{holomorphy ring} if there exists a finite subset $S \subset S_{\text{fin}}$ (and $\emptyset \ne S$ in the function field case) such that
\[
\cO = \cO_S = \bigcap_{v \in S_{\text{fin}}\setminus S} \cO_v.
\]
The holomorphy rings in $K$ are Dedekind domains which are properly contained in $K$ and have quotient field $K$.
The most important examples are rings of algebraic integers and $S$-integers in the number field case, and coordinate rings of non-singular irreducible affine algebraic curves over finite fields in the function field case.

Let $A$ be a central simple $K$-algebra, that is, a finite-dimensional $K$-algebra with center $K$ which is simple as a ring.
A \emph{classical $\cO$-order} is a subring $\cO \subset R \subset A$ such that $R$ is a finitely generated $\cO$-module and $KR = A$.
A \emph{classical maximal $\cO$-order} is a classical $\cO$-order which is maximal with respect to set inclusion within the set of all classical $\cO$-orders contained in $A$.
A \emph{classical hereditary $\cO$-order} is a classical $\cO$-order which is hereditary as a ring.
Every classical maximal $\cO$-order is hereditary.

If $v$ is a place of $K$, the completion $A_v$ of $A$ is a central simple algebra over the completion $K_v$ of $K$.
Hence, $A_v$ is of the form $A_v \cong M_{n_v}(D_v)$ with a finite-dimensional division ring $D_v \supset K_v$.
The algebra $A$ is \emph{ramified at $v$} if $D_v \ne K_v$.

\runinhead{Isomorphism classes of right ideals and class groups.}
Let $\cF^\times(\cO)$ denote the group of nonzero fractional ideals of $\cO$.
Let $K_A^\times$ denote the subgroup of $K^\times$ consisting of all $a \in K^\times$ for which $a_v > 0$ for all archimedean places $v$ of $K$ at which $A$ is ramified.
To a classical maximal $\cO$-order $R$ (or more generally, a classical hereditary $\cO$-order), we associate the ray class group
\[
\begin{split}
  \Cl_A(\cO) = \cF^\times(\cO) \,/\, \{\, a\cO \mid a \in K_A^\times \,\}.
\end{split}
\]
This is a finite abelian group, with operation induced by the multiplication of fractional ideals.

Let $\LF_1(R)$ denote the (finite) set of isomorphism classes of right $R$-ideals.
In general, $\LF_1(R)$ does not have a natural group structure.
Let $\cC(R)$ denote the set of stable isomorphism classes of right $R$-ideals.
The set $\cC(R)$ naturally has the structure of an abelian group, with operation induced from the direct sum operation.
There is a surjective map of sets $\LF_1(R) \to \cC(R)$, and a group homomorphism $\cC(R) \to \Cl_A(\cO), [I] \mapsto [\nr(I)]$.
The homomorphism $\cC(R) \to \Cl_A(\cO)$ is in fact an isomorphism (see \cite[Corollary 9.5]{swan80}).
However, the map $\LF_1(R) \to \cC(R)$ need not be a bijection in general.
It is a bijection if and only if stable isomorphism of right $R$-ideals implies isomorphism.
This holds if $A$ satisfies the Eichler condition relative to $\cO$ (see below).
We will at some point need to impose the weaker condition that every stably free right $R$-ideal is free, that is, that the preimage of the trivial class under $\LF_1(R) \to \cC(R)$ consists only of the trivial class.
This condition will be of paramount importance for the existence of a transfer homomorphism from $R^\bullet$ to a monoid of zero-sum sequences over the ray class group $\cC_A(\cO)$.

A ring over which every finitely generated stably free right module is free is called a \emph{(right) Hermite ring}.
(Using the terminology of \cite[Chapter~I.4]{lam06}, some authors require in addition that $R$ has the invariant basis number (IBN) property. For instance, this is the case in \cite[Chapter~0.4]{cohn06}.)
For a classical maximal $\cO$-order $R$, every finitely generated projective right $R$-module is of the form $R^n \oplus I$ for a right ideal $I$ of $R$.
It follows that $R$ is a Hermite ring if and only if every stably free right $R$-ideal is free.

\runinhead{Strong approximation and Eichler condition.}
Let $S \subset S_{\text{fin}}$ be the set of places defining the holomorphy ring $\cO=\cO_S$.
Denote by $S_\infty$ the set of archimedean places of $K$.
($S_\infty=\emptyset$ if $K$ is a function field.)
We consider the places in $S_{\text{fin}}\setminus S$ to be places arising from $\cO$, since they correspond to maximal ideals of $\cO$.
We consider the places of $S_\infty \cup S$ to be places not arising from $\cO$.
The algebra $A$ satisfies the \emph{Eichler condition (relative to $\cO$)} if there exists a place $v$ not arising from $\cO$ such that $A_v$ is not a noncommutative division ring.

If $K$ is a number field, and $A$ does not satisfy the Eichler condition, then $A$ is necessary a totally definite quaternion algebra.
That is, $\dim_K A = 4$ and, for all $v \in S_{\infty}$, we have $K_v \cong \bR$ and $A_v$ is a division ring, necessarily isomorphic to the Hamilton quaternion algebra.

The Eichler condition is a sufficient condition to guarantee the existence of a strong approximation theorem for the kernel of the reduced norm, considered as a homomorphism of the idele groups.
As a consequence, if $A$ satisfies the Eichler condition, then the map $\LF_1(R) \to \cC(R)$ is a bijection.
In particular, every stably free right $R$-ideal is free.
(See \cite{reiner75,swan80,curtis-reiner87}.)

On the other hand, if $K$ is a number field, $\cO$ is its ring of algebraic integers, and $A$ is a totally definite quaternion algebra, then, for all but finitely many isomorphism classes of $A$ and $R$, there exist stably free right $R$-ideals which are not free.
The classical maximal orders for which this happens have been classified.
(See \cite{vigneras76,hallouin-maire06,smertnig-canc}.)

The strong approximation theorem is also useful in the determination of the image of the reduced norm of an order.
Suppose that $A$ satisfies the Eichler condition with respect to $\cO$.
Let $\cO_A^\bullet$ denote the subsemigroup of all nonzero elements of $\cO$ which are positive at each $v \in S_\infty$ which ramifies in $A$.
Then, if $R$ is a classical hereditary $\cO$-order in $A$, the strong approximation theorem together with an explicit characterization of local hereditary orders implies that $\nr(R^\bullet)=\cO_A^\bullet$.
(See \cite[Theorem 39.14]{reiner75} for the classification of hereditary orders in a central simple algebra over a quotient field of a complete DVR, and \cite[Theorem 8.2]{swan80} or \cite[Theorem 52.11]{curtis-reiner87} for the globalization argument via strong approximation.)

\runinhead{Hurwitz quaternions.}
Historically, the order of Hurwitz quaternions has received particular attention.
It is Euclidean, hence a PID, and therefore enjoys unique factorization in a sense.
An elementary discussion of the Hurwitz quaternions (without reference to the theory of maximal orders) and their factorization theory can be found in \cite{conway-smith03}.
We give \cite{vigneras76,maclachlan-reid03} as references for the theory of quaternion algebras over number fields.
\begin{example}
  Let $K$ be a field of characteristic not equal to $2$.
  Usually, we will consider $K=\bQ$ or $K=\bR$.
  Let $\bH_K$ denote the four-dimensional $K$-algebra with basis $1$, $i$, $j$, $k$, where $i^2=j^2=-1$, $ij=-ji=k$, and $1$ is the multiplicative identity.
  This is a quaternion algebra, that is, a four-dimensional central simple $K$-algebra.
  On $\bH_K$ there exists an involution, called \emph{conjugation}, defined by $K$-linear extension of $\overline{1}=1$, $\overline{i}=-i$, $\overline{j}=-j$, and $\overline{k}=-k$.
  The \emph{reduced norm} $\nr\colon H_K \to K$ is defined by $\nr(x)=x\overline x$ for all $x \in \bH_K$.
  Thus $\nr(a+bi+cj+dk) = a^2+b^2+c^2+d^2$ if $a$,~$b$,~$c$,~$d \in K$.
  If $K=\bR$, then $\bH_K$ is the division algebra of \emph{Hamilton quaternions}.

  The algebra $\bH_\bQ$ is a totally definite quaternion algebra over $\bQ$.
  Let $\cH$ be the classical $\bZ$-order with $\bZ$-basis $1$, $i$, $j$, $\frac{-1+i+j+k}{2}$ in $\bH_\bQ$.
  That is, $\cH$ consists of elements $a + b i + c j + d k $ with $a$, $b$, $c$, $d$ either all integers or all half-integers.
  Then $\cH$ is a classical maximal $\bZ$-order, the order of \emph{Hurwitz quaternions}.
  The ring $\cH$ is Euclidean with respect to the reduced norm, and hence a PID.

  The unit group of $\cH$ consists of the 24 elements
  \[
  \cH^\times = \bigg\{ \pm 1, \pm i, \pm j, \pm k, \frac{\pm 1 \pm i \pm j \pm k}{2} \bigg\}.
  \]
  Up to conjugation by units of $\bH_\bQ$, the order of Hurwitz quaternions is the unique classical maximal $\bZ$-order in $\bH_\bQ$.
  The algebra $\bH_\bQ$ is only ramified at $2$ and $\infty$.
  Thus, for any odd prime number $p$, one has $\bH_\bQ \otimes_\bQ \bQ_p \cong M_2(\bQ_p)$ and $\cH \otimes_\bZ \bZ_p \cong M_2(\bZ_p)$.
  Moreover, in this case, $\cH/p\cH \cong M_2(\bF_p)$.

  On the other hand, $\bH_\bQ \otimes_\bQ \bR \cong \bH_\bR$ is a division algebra.
  Similarly, for $p=2$, the completion $\bH_\bQ \otimes_\bQ \bQ_2$ is isomorphic to the unique quaternion division algebra over $\bQ_2$.

  In the maximal order $\cH \otimes_\bZ \bZ_2$, every right or left ideal is two-sided.
  The ideals of $\cH \otimes_\bZ \bZ_2$ are linearly ordered, and each of them is a power of the unique maximal ideal, which is generated by $(1+i)$.
  Note that this is not the case for $p$ odd, since then $\cH \otimes_\bZ \bZ_p \cong M_2(\bZ_p)$.
\end{example}

\section{Factorizations and Atomicity}
\label{sec:rf}

We develop the basic notions of (rigid) factorizations in the very general setting of a cancellative small category.
Moreover, we show how this notion is connected to chains of principal right ideals and recall sufficient conditions for a cancellative small category to be atomic.

We introduce the notions for a cancellative small category $H$.
When we later apply them to a ring $R$, we implicitly assume that they are applied to the semigroup of non-zero-divisors $R^\bullet$.
For instance, when we write ``$R$ is atomic'', this means ``$R^\bullet$ is atomic'', and so on.

\subsection{Rigid Factorizations}

\begin{center}
\emph{Let $H$ be a cancellative small category.}
\end{center}

\begin{definition}
An element $a \in H$ is an \emph{atom} if $a=bc$ with $b$,~$c \in H$ implies $b \in H^\times$ or $c \in H^\times$.
\end{definition}

Viewing $H$ as a quiver (a directed graph with multiple edges allowed), the atoms of $H$ form a subquiver, denoted by $\cA(H)$.
We will often view $\cA(H)$ simply as a set of atoms, forgetting about the additional quiver structure.

A \emph{rigid factorization} of $a \in H$ is a representation of $a$ as a product of atoms up to a possible insertion of units.
We first give an informal description.
We write the symbol $\rfop$ between factors in a rigid factorizations, to distinguish the factorization as a formal product from its actual product in $H$.
Thus, if $a \in H$ and $a=\varepsilon_1 u_1\cdots u_k$ with atoms $u_1$, $\ldots\,$,~$u_k$ of $H$ and $\varepsilon_1 \in H^\times$, then $z=\rf[\varepsilon_1]{u_1,\ldots,u_k}$ is a rigid factorization of $a$.
If $\varepsilon_2$,~$\ldots\,$,~$\varepsilon_k \in H^\times$ are such that $t(\varepsilon_i)=s(u_{i})$, then also $z=\rf{\varepsilon_1 u_1\varepsilon_2^{-1}, \varepsilon_2 u_2 \varepsilon_3^{-1},\ldots, \varepsilon_k u_k}$ represents the same rigid factorization of $a$.
The unit $\varepsilon_1$ can be absorbed into $u_1$, unless $k=0$, that is, unless $a \in H^\times$.

If $a$,~$b \in H$ and $t(a)=s(b)$, then two rigid factorization $z$ of $a$ and $z'$ of $b$ can be composed in the obvious way to obtain a rigid factorization of $ab$.
We write $z \rfop z'$ for this composition.
In this way, the rigid factorizations themselves form a cancellative small category, denoted by $\sZ^*(H)$.

More formally, we make the following definitions.
See \cite[Section 3]{smertnig13} or \cite[Section 3]{baeth-smertnig15} for details.
Let $\cF^*(\cA(H))$ denote the path category on the quiver $\cA(H)$.
Thus, $\cF^*(\cA(H))_0=H_0$.
Elements (\emph{paths}) $x \in \cF^*(\cA(H))$ are denoted by
\[
x=(e,u_1,\ldots,u_k,f)
\]
where $e$,~$f \in H_0$, and $u_i \in \cA(H)$ with $s(u_1)=e$, $t(u_k)=f$, and $t(u_i)=s(u_{i+1})$ for $i \in [1,k-1]$.
We set $s(x)=e$, $t(x)=f$, and the composition is given by the obvious concatenation of paths.

Denote by $H^\times \times_r \cF^*(\cA(H))$ the cancellative small category
\[
H^\times \times_r \cF^*(\cA(H)) = \big\{\, (\varepsilon,x) \in H^\times \times \cF^*(\cA(H)) \mid t(\varepsilon)=s(x) \,\big\},
\]
where $\big(H^\times \times_r \cF^*(\cA(H))\big)_0 = \{\, (e,e) \mid e \in H_0 \,\}$, which we identify with $H_0$, $s((\varepsilon,x))=s(\varepsilon)$ and $t((\varepsilon,x)) = t(x)$.
If $x=(e,u_1,\ldots,u_k,f)$, $y=(e',v_1,\ldots,v_l,f')$ in $\cF^*(\cA(H))$ and $\varepsilon$,~$\varepsilon' \in H^\times$ are such that $(\varepsilon,x)$,~$(\varepsilon',y) \in H^\times \times_r \cF^*(\cA(H))$ with $t(x)=s(\varepsilon')$, we set
\[
(\varepsilon, x) (\varepsilon', y) = (\varepsilon, (e,u_1,\ldots, u_k \varepsilon', v_1,\ldots, v_l f') \quad\text{if $k>0$,}
\]
and $(\varepsilon,x)(\varepsilon',y)= (\varepsilon \varepsilon', y)$ if $k=0$.

On $H^\times \times_r \cF^*(\cA(H))$ we define a congruence relation $\sim$ by $(\varepsilon,x) \sim (\varepsilon',y)$ if and only if
\begin{enumerate}
  \item $k=l$,
  \item $\varepsilon u_1 \cdots u_k = \varepsilon' v_1 \cdots v_l \in H$, and
  \item there exist $\delta_2$, $\ldots\,$,~$\delta_k \in H^\times$ and $\delta_{k+1}=t(u_k)$, such that
    \[
    \varepsilon' v_1 = \varepsilon u_1 \delta_2^{-1} \quad\text{and}\quad v_i = \delta_i u_i \delta_{i+1}^{-1} \quad\text{for all $i \in [2,k]$.}
    \]
\end{enumerate}

\begin{definition}
The quotient category $\sZ^*(H) = H^\times \times_r \cF^*(\cA(H)) / \sim$ is called the \emph{category of rigid factorizations} of $H$.
The class of $(\varepsilon,x)$ (as above) in $\sZ^*(H)$ is denoted by $\rf[\varepsilon]{u_1,\ldots,u_k}$.
There is a natural homomorphism
\[
\pi=\pi_H\colon \sZ^{*}(H) \to H,\quad \rf[\varepsilon]{u_1,\ldots,u_k} \mapsto \varepsilon u_1\cdots u_k.
\]
For $a \in H$, the set $\sZ^*(a) = \sZ_H^*(a) = \pi^{-1}(a)$ is the set of \emph{rigid factorizations of $a$}.
If $z=\rf[\varepsilon]{u_1,\ldots,u_k} \in \sZ^*(H)$, then $\length{z}=k$ is the \emph{length} of the (rigid) factorization $z$.
\end{definition}

\begin{remark}
  \begin{enumerate}
  \item
    If $H$ is a cancellative semigroup, then $H^\times \times_r \cF^*(\cA(H))$ is the product of $H^\times$ and the free monoid on $\cA(H)$.
    If moreover $H$ is reduced, then $\sZ^*(H)$ is the free monoid on $\cA(H)$.
    Hence, in this case, rigid factorizations are simply formal words on the atoms of $H$.
    In particular, if $H$ is a reduced commutative cancellative semigroup, we see that rigid factorizations are ordered, whereas the usual notion of factorizations is unordered.

  \item
    While complicating the definitions a bit, the presence of units in the definition of $\sZ^*(H)$ allows for a more uniform treatment of factorizations.
    It often makes it unnecessary to treat units as a (trivial) special case.
    In particular, with our definitions, every unit has a unique (trivial) rigid factorization of length $0$.
  \end{enumerate}
\end{remark}

\subsection{Factor Posets}

\begin{center}
  \emph{Let $H$ be a small category.}
\end{center}

Another useful way of viewing rigid factorizations is in terms of chains of principal left or right ideals.
Suppose that, for $a$,~$b \in H^\bullet$, we have $aH \subset bH$ if and only if there exists $c \in H^\bullet$ such that $a=bc$.
\footnote{We may always force this condition by replacing $H$ by the subcategory of all cancellative elements.
Note that then principal right ideals $aH$ have to be replaced by $aH^\bullet$.
Sometimes it can be more convenient work with $H$ with $H^\bullet \ne H$, because typically we will have $H=R$ a ring and $H^\bullet = R^\bullet$ the semigroup of non-zero-divisors.
In this setting, sufficient conditions for the stated condition to be satisfied are for $R^\bullet$ to be Ore, or $R$ to be a domain.}
If $a \in H^\bullet$ and $b \in H^\bullet$, then $aH = bH$ if and only if there exists an $\varepsilon \in H^\times$ such that $a=b\varepsilon$, that is, $a$ and $b$ are right associated.

For $a \in H^\bullet$, let
\[
\iv{aH}{H} = \big\{\, bH \mid b \in H^\bullet\text{ such that } aH \subset bH \subset H \,\big\}
\]
denote the set of all principal right ideals containing $aH$ which are generated by a cancellative element.
Note that $\iv{aH}{H}$ is naturally a partially ordered set via set inclusion.
This order reflects left divisibility in the following sense:
Left divisibility gives a preorder on the cancellative left divisors of $a$.
The corresponding poset, obtained by identifying right associated cancellative left divisors of $a$, is order anti-isomorphic to $\iv{aH}{H}$.
We call $\iv{aH}{H}$ the \emph{(right) factor poset} of $a$.

An element $a \in H^\bullet$ is an atom if and only if $[aH,H]=\{aH, H\}$.
Rigid factorizations of $a$, that is, elements of $\sZ^*(a)$, are naturally in bijection with finite maximal chains in $\iv{aH}{H}$.
For instance, a rigid factorization $z=\rf{u_1,\ldots,u_k}$ of $a$ corresponds to the chain
\[
aH=u_1\cdots u_kH \subsetneq u_1\cdots u_{k-1}H \subsetneq \ldots \subsetneq u_1u_2H \subsetneq u_1H \subsetneq H.
\]
Thus, naturally, properties of the set of rigid factorizations of $a$ correspond to properties of the poset $\iv{aH}{H}$.

In particular, we are interested in $\iv{aH}{H}$ being a lattice (or, stronger, a sublattice of the right ideals of $H$).
If the factor poset $\iv{aH}{H}$ is a lattice, we are interesting in it being (semi-)modular or distributive.
For a modular lattice the Schreier refinement theorem holds: Any two chains have equivalent refinements.
For semimodular lattices of finite length one has a Jordan-Hölder Theorem (and finite length of a semimodular lattice is already guaranteed by the existence of one maximal chain of finite length).
Thus, if all factor posets are (semi-)modular lattices, we obtain unique factorization results for elements.
This point of view will be quite useful in understanding and reconciling results on unique factorization in various classes of rings, such as $2$-firs, modular LCM domains, and LCM domains having RAMP (see Section~\ref{ssec:similarity-ufds}).

\begin{remark}
Given an element $a \in H^\bullet$, we have defined $\iv{aH}{H}$ in terms of principal right ideals of $H$.
We may similarly define $\iv{Ha}{H}$ using principal left ideals.
If $b \in H^\bullet$ and $bH \in \iv{aH}{H}$, then there exists $b' \in H^\bullet$ such that $a=bb'$.
This element $b'$ is uniquely determined by $bH$ up to left associativity, that is, $Hb'$ is uniquely determined by $bH$.
Hence, there is an anti-isomorphism of posets
\[
\iv{aH}{H} \to \iv{Ha}{H}, \quad bH \mapsto Hb'.
\]
\end{remark}

\subsection{Atomicity, BF-Categories, and FF-Categories}

\begin{center}
  \emph{Let $H$ be a cancellative small category.}
\end{center}

\begin{definition}
  \mbox{}
  \begin{enumerate}
    \item $H$ is \emph{atomic} if the set if rigid factorizations, $\sZ^*(a)$, is non-empty for all $a \in H$.
      Explicitly, for every $a \in H$, there exist $k \in \bN_0$, atoms $u_1$,~$\ldots\,$,~$u_k \in \cA(H)$, and a unit $\varepsilon \in H^\times$ such that $a=\varepsilon u_1\cdots u_k$.

    \item $H$ is a \emph{BF-category} (a category with \emph{bounded factorizations}) if the set of lengths, $\sL(a) = \{\, \length{z} \mid z \in \sZ^*(a) \,\}$, is non-empty and finite for all $a \in H$.

    \item $H$ is \emph{half-factorial} if $\card{\sL(a)}=1$ for all $a \in H$.

    \item $H$ is an \emph{FF-category} (a category with \emph{finite factorizations}) if the set of rigid factorizations, $\sZ^*(a)$, is non-empty and finite for all $a \in H$.
  \end{enumerate}
\end{definition}
Obviously, any FF-category is a BF-category.
Analogous definitions are made for BF-semigroups, BF-domains, etc., and FF-semigroups, FF-domains, etc.
\begin{remark}
The definition of an FF-category here is somewhat ad hoc in that it relates only to rigid factorizations, but this is in line with \cite{bell-heinle-levandovskyy}.
It is a bit restrictive in that a PID need not be an FF-domain (see Example~\ref{e-pid-not-ff}).
It may be more accurate to talk of a \emph{finite rigid factorizations} category.
\end{remark}

The following condition for atomicity is well known.
A proof can be found in \cite[Lemma 3.1]{smertnig13}.
\begin{lemma}
  If $H$ satisfies both, the ACC on principal left ideals and the ACC on principal right ideals, then $H$ is atomic.
\end{lemma}
\begin{remark}
  Suppose for a moment that $H$ is a small category which is not necessarily cancellative.
  If $H$ satisfies the ACC on right ideals generated by cancellative elements, then $H^\bullet$ satisfies the ACC on principal right ideals.
  (If $a$,~$b \in H^\bullet$ with $aH=bH$, then $a$ and $b$ are right associated, and hence also $aH^\bullet=bH^\bullet$.)
  Hence $H^\bullet$ is atomic.
  Phrasing the condition in this slightly more general way is often more practical.
  For instance, if $R$ is a Noetherian ring, then $R^\bullet$ is atomic.
\end{remark}

A more conceptual way of looking at the previous lemma is the following.
By the duality of factor posets, the ACC on principal left ideals is equivalent to the restricted DCC on principal right ideals.
That is, the ACC on principal left ideals translates into the DCC on $[aH,H]$ for $a \in H$.
Thus, $[aH,H]$ has the ACC and DCC.
Hence, there exist maximal chains in $[aH,H]$ and any such chain $[aH,H]$ has finite length.
From this point of view, it is not surprising that the ACC on principal right ideals by itself is not sufficient for atomicity, as the following example shows.

\begin{example} \label{e-rightpid-notatomic}
  A domain $R$ is a \emph{right Bézout domain} if every finitely generated right ideal of $R$ is principal.
  $R$ is a \emph{Bézout domain} if it is both, a left and right Bézout domain.
  Trivially, every PID is a Bézout domain.

  Let $R$ be a Bézout domain which is a right PID but not a left PID.
  (Such a domain, which is moreover simple, was constructed by P.\,M.~Cohn and Schofield in \cite{cohn-schofield85}.)
  Then $R$ does not satisfy the ACC on principal left ideals. (For otherwise it would satisfy the ACC on finitely generated left ideals, and hence be left Noetherian. This would in turn imply that it is a left PID.)
  However, an atomic Bézout domain satisfies the ACC on principal left ideals and the ACC on principal right ideals.
  (This follows from the Schreier refinement theorem.)
  Hence $R$ is not atomic.
\end{example}

A function $\ell\colon H \to \bN_0$ is called a \emph{(right) length function} if $\ell(a) > \ell(b)$ whenever $a=bc$ with $b$,~$c \in H$ and $c \notin H^\times$.
If $H$ has a right length function, then it is easy to see that $H$ satisfies the ACC on principal right ideals, as well as the restricted DCC on principal right ideals.
In fact, if $H$ has a right length function, then $[aH,H]$ has finite length for all $a \in H$.
Thus, the length of a factorization of $a$ is bounded by $\ell(a)$, and we have the following.

\begin{lemma} \label{l-length-bf}
  If $H$ has a right length function, then $H$ is a BF-category.
\end{lemma}

\section{Unique Factorization}
\label{sec:unique-factorization}

It turns out to be non-trivial to obtain a satisfactory theory of factorial domains (also called unique factorization domains, short UFDs) in a noncommutative setting.
Many different notions of factoriality have been studied.
They cluster into two types.

First, there are definitions based on an element-wise notion of the existence and uniqueness of factorizations.
For such a definition, typically, every non-zero-divisor has a factorization which is in some sense unique up to order and an equivalence relation on atoms.
Usually, such classes of rings will contain PIDs but will not be closed under some natural ring-theoretic constructions, such as forming a polynomial ring or a ring of square matrices.
This will be the focus of Section~\ref*{ssec:similarity-ufds}.

Second, definitions have been studied which start from more ring-theoretic characterizations of factorial commutative domains.
Here, one does not necessarily obtain element-wise unique factorization results.
Instead, one has unique factorization for normal elements into normal atoms.
On the upside, this type of definition tends to behave better with respect to natural ring-theoretic constructions.
This will be discussed in Section~\ref*{ssec:chatters-ufds}.

\subsection{Similarity Factorial Domains and Related Notions}
\label{ssec:similarity-ufds}

We first discuss the notions of similarity factoriality and $n$-firs.
These have mainly been studied by P.\,M.~Cohn and Bergman.
(Although it seems that Bergman did not publish most of the results outside of his thesis \cite{bergman68}.)
We mention as general references for this section \cite{cohn85,cohn06,bergman68} as well as the two surveys \cite{cohn63,cohn63b} and \cite{cohn73,cohn73-c}.

Brungs, in \cite{brungs69}, introduced the weaker notion of subsimilarity factorial domains.
This permits a form of Nagata's theorem to hold.
Beauregard has investigated right LCM domains and the corresponding notion of projectivity factoriality.
These works will also be discussed in this section.

Let $R$ be a domain and $a$,~$b \in R^\bullet$.
We call $a$ and $b$ \emph{similar} if $R/aR \cong R/bR$ as right $R$-modules.
Fitting, in \cite{fitting36}, observed that $R/aR \cong R/bR$ if and only if $R/Ra \cong R/Rb$, and hence the notion of similarity is independent of whether we consider left or right modules.
(This duality has later been extended to the factorial duality by Bergman and P.\,M.~Cohn, see \cite{cohn73} or \cite[Theorem~3.2.2]{cohn06}.)

If $R$ is commutative, and $R/aR \cong R/bR$ for $a$,~$b \in R$, then we have $aR = \ann(R/aR) = \ann(R/bR) = bR$, and thus $a$ and $b$ are similar if and only if they are associated.
For noncommutative domains it is no longer true in general that $R/aR\cong R/bR$ implies that $a$ and $b$ are left-, right-, or two-sided associated.
\begin{definition} \label{def:similarity-factorial}
  A domain $R$ is called \emph{similarity factorial} (or, a \emph{similarity-UFD}) if
  \begin{enumerate}
    \item $R$ is atomic, and
      \item if $u_1\cdots u_m = v_1\cdots v_n$ for atoms $u_1$, $\ldots\,$,~$u_m$, $v_1$, $\ldots\,$,~$v_n \in R$, then $m=n$ and there exists a permutation $\sigma \in \fS_m$ such that $u_i$ is similar to $v_{\sigma(i)}$ for all $i \in [1,m]$.
  \end{enumerate}
\end{definition}
\begin{remark}
  \begin{enumerate}
    \item
      A note on terminology.
      It is more common to refer to similarity factorial domains as \emph{similarity-UFDs}.
      P.\,M.~Cohn calls a similarity-UFD simply a UFD.
      We use the terminology \emph{similarity factorial domains}, because using the adjective ``factorial'' over the noun ``UFD'' is more in line with the modern development of the terminology in the commutative setting.

      In \cite{baeth-smertnig15}, a similarity factorial domain is called \emph{$\sd_{\text{sim}}$-factorial}.
      This follows a general system: In \cite{baeth-smertnig15}, distances between rigid factorizations are introduced.
      Each distance $\sd$ naturally gives rise to a corresponding notion of $\sd$-factoriality by identifying two rigid factorizations of an element if they have distance $0$.
      The distance $\sd_{\text{sim}}$ is defined using the similarity relation.
      See Section~\ref{ss-arithmetical-invariants} below for more on this point of view.

    \item Let $R$ be a ring which is not necessarily a domain.
      We call $R$ \emph{(right) similarity factorial} if $R^\bullet$ is atomic, and factorizations of elements in $R^\bullet$ are unique up to order and similarity of the atoms.
      In general it is no longer true that right and left similarity are the same.
    \end{enumerate}
\end{remark}

\begin{example} \label{e-sim-ufd}
  \begin{enumerate}
  \item\label{e-sim-ufd:pid}
    Every PID is similarity factorial. This is immediate from the Jordan-Hölder Theorem.

  \item\label{e-sim-ufd:fir}
    Let $K$ be a field.
    In the free associative $K$-algebra $R = K\langle x,y\rangle$, the elements $x$ and $y$ are similar but not associated.
    We will see below that $K\langle x,y\rangle$ is similarity factorial.
    However, factorizations are not unique up to order and associativity, as
    \[
    x(yx+1) = (xy+1)x
    \]
    shows.

  \item \label{e-sim-ufd:maxord}
    Let $R$ be a classical maximal $\bZ$-order in a definite quaternion algebra over $\bQ$.
    Suppose that $R$ is a PID.
    Then $R$ is similarity factorial.
    For every prime number $p$ which is unramified in $R$, there exist $p+1$ atoms with reduced norm $p$.
    These $p+1$ atoms are all similar, but, since $R^\times$ is finite, for sufficiently large $p$, they cannot all be right-, left-, or two-sided associated.
    For instance, this is the case for $R=\cH$, the ring of Hurwitz quaternions.
  \end{enumerate}
\end{example}

One may be tempted to require factorizations to be unique up to order and, say, two-sided associativity of elements.
This is referred to as \emph{permutably factorial} in \cite{baeth-smertnig15}.
However, Examples~\ref*{e-sim-ufd:fir} and \ref*{e-sim-ufd:maxord} above show that such a notion is often too restrictive.

If $R$ is a PID, then $R$ is similarity factorial.
However, when looking for natural examples of similarity factorial domains, one should consider a more general class of rings than PIDs, namely that of $2$-firs.
The motivation for this is the following:
If $K$ is a field and $R=K\langle x,y\rangle$ is the free associative $K$-algebra in two indeterminates, then $xR \cap yR = \mathbf 0$.
Hence $xR + yR \cong R^2$ is a non-principal right ideal of $R$.
Thus $R$ is not a PID.
However, P.\,M.~Cohn has shown that $R$ is an atomic $2$-fir and hence, in particular, similarity factorial (see below).

\begin{definition}
  Let $n \in \bN$.
  A ring $R$ is an \emph{$n$-fir} if every right ideal of $R$ on at most $n$ generators is free, of unique rank.
  A ring $R$ is a \emph{semifir} if $R$ is an $n$-fir for all $n \in \bN$.
\end{definition}

It can be shown that the notion of an $n$-fir for $n \in \bN$ is symmetric (see \cite[Theorem 2.3.1]{cohn06}).
Thus $R$ is an $n$-fir if and only if every left ideal of $R$ on at most $n$ generators is free, of unique rank.
Any $n$-fir is of course an $m$-fir for all $m < n$.
A ring $R$ is a \emph{right fir} (free right ideal ring) if all right ideals of $R$ are free, of unique rank.
$R$ is a \emph{fir} if it is a left and right fir.
Any fir is atomic (see \cite[Theorem~2.2.3]{cohn06}).

The case which is particularly important for the factorization of elements is that of a $2$-fir.
(More generally, over a $2n$-fir one can consider factorizations of $n \times n$-matrices.)
A ring $R$ is a $1$-fir if and only if it is a domain.
Thus, in particular, any $2$-fir is a domain.

\begin{theorem}[{\cite[Theorem 2.3.7]{cohn06}}] \label{t-2fir}
  For a domain $R$, the following conditions are equivalent.
  \begin{equivenumerate}
    \item $R$ is a $2$-fir.
    \item For $a$,~$b \in R^\bullet$ we have $aR \cap bR = mR$ for some $m \in R$, while $aR + bR$ is principal if and only if $m \ne 0$.
    \item \label{t-2fir:sum} If $a$,~$b\in R$ are such that $aR \cap bR \ne \mathbf 0$, then $aR + bR$ is a principal right ideal of $R$.
    \item \label{t-2fir:lat} For all $a \in R^\bullet$,  $[aR, R]$ is a sublattice of the lattice of all right ideals of $R$.
  \end{equivenumerate}
\end{theorem}

It follows from \ref*{t-2fir:sum}, that a $2$-fir is a right Ore domain if and only if it is a right Bézout domain.
In particular, a commutative ring is a $2$-fir if and only if it is a Bézout domain.

Note that \ref*{t-2fir:lat} implies that $[aR,R]$ is a modular lattice for all $a \in R^\bullet$.
The Schreier refinement theorem for modular lattices then implies that finite maximal chains of $[aR,R]$ are unique up to perspectivity.
In particular, if $[aR,R]$ contains any finite maximal chain, then $[aR,R]$ has finite length.

Since $[aR,R]$ is a sublattice of the lattice of right ideals of $R$, the uniqueness of maximal chains up to perspectivity translates into the factors of a maximal chain being isomorphic as modules (up to order).
Translated into factorizations, this implies that the factorizations of nonzero elements in $R$ are unique up to order and similarity.
More generally, one obtains a similar result for factorizations of full matrices in $M_n(R)$ over a $2n$-fir $R$.
A matrix $A \in M_n(R)$ is \emph{full} if it cannot be written in the form $A=BC$ with $B$ an $n\times r$-matrix and $C$ and $r\times n$-matrix where $r < n$.
Over an $n$-fir, any full matrix $A \in M_n(R)$ is cancellative (see \cite[Lemma 3.1.1]{cohn06}).
A \emph{full atom} is a (square) full matrix which cannot be written as a product of two non-unit full matrices.
\begin{theorem}[{\cite[Chapter 3.2]{cohn06}}] \label{t-nfir-sf}
  If $R$ is a $2n$-fir, any two factorization of a full matrix in $M_n(R)^\bullet$ into full atoms are equivalent up to order and similarity of the atoms.
  In particular, if $R$ is an atomic $2$-fir, then $R$ is similarity factorial.
\end{theorem}

\begin{remark}
  \begin{enumerate}
  \item A commutative atomic $2$-fir is an atomic Bézout domain, and hence a PID.
    However, noncommutative atomic $2$-firs need not be PIDs.
    The free associative algebra $K\langle x, y \rangle$ over a field $K$ provides a counterexample.
  \item If $R$ is a semifir, then products of full matrices are full (see \cite[Corollary 5.5.2]{cohn06}), so that the full matrices form a subsemigroup of $M_n(R)^\bullet$.
  \item Let $R$ be a commutative Noetherian ring with no nonzero nilpotent elements.
    If $M_n(R)$ is similarity factorial for all $n \ge 2$ (equivalently, $M_2(R)$ is similarity factorial), then $R$ is a finite direct product of PIDs (see \cite{estes-matijevic79b} or Theorem~\ref{thm:em-matrix-sim-fact}).
    This is a partial converse to the theorem above.

  \item
    Leroy and Ozturk, in \cite{leroy-ozturk04}, introduced F-algebraic and F-independent sets to study factorizations in $2$-firs.
    In particular, they obtain lower bounds on the lengths of elements in terms of dimensions of certain vector spaces.
  \end{enumerate}
\end{remark}

A sufficient condition for a domain to be an atomic right PID, respectively an atomic $n$-fir, is the existence of a right Euclidean algorithm, respectively an $n$-term weak algorithm.

A domain $R$ is right Euclidean if there exists a function $\delta\colon R \to \bN_0 \cup \{-\infty\}$ such that, for all $a$,~$b \in R$, if $b \ne 0$, there exist $q$,~$r \in R$ such that $a = bq + r$ and $\delta(r) < \delta(b)$.
Equivalently, if $a$,~$b \in R$ with $b \ne 0$, and $\delta(b) \le \delta(a)$, then there exists $c \in R$ such that \begin{equation}\label{e-euclidean}
\delta(a - bc) < \delta(a).
\end{equation}
Any right Euclidean domain is a right PID and moreover atomic.
Thus, right Euclidean domains are similarity factorial.
The atomicity follows since the least function defining the Euclidean algorithm induces a right length function on $R^\bullet$ (see \cite[Proposition 1.2.5]{cohn06}).
By contrast, we recall that a right PID need not be atomic (see Example~\ref{e-rightpid-notatomic}).
See \cite[\S3.2.7]{berrick-keating00} for a discussion of Euclidean domains.
An extensive discussion of Euclidean rings can be found in \cite[Chapter 1.2]{cohn06}.

\begin{example}
  \begin{enumerate}
  \item Let $D$ be a division ring, $\sigma$ an injective endomorphism of $D$ and $\delta$ a (right) $\sigma$-derivation (that is, $\delta(ab)=\delta(a)\sigma(b) + a \delta(b)$ for all $a$,~$b \in D$).
    The skew polynomial ring $D[x;\sigma,\delta]$ consists of elements of the form
    \[
    \sum_{n \in \bN_0} x^n a_n \qquad\text{with $a_n \in D$, almost all zero.}
    \]
    The multiplication is defined by $ax = x\sigma(a) + \delta(a)$.
    We set $D[x;\sigma]=D[x;\sigma,0]$ and $D[x;\delta]=D[x;\id_D,\delta]$ if $\delta$ is a derivation.

    Using polynomial division, it follows that $D[x;\sigma,\delta]$ is right Euclidean with respect to the degree function.
    If $\sigma$ is an automorphism, then $D[x;\sigma,\delta]$ is also left Euclidean, by symmetry.

    In particular, if $K$ is a field and $x$ is an indeterminate, then $B_1(K)=K(x)[y;-\frac{d}{dx}]$ is Euclidean.
    If the characteristic of $K$ is $0$, then $K(x)$ naturally has a faithful right $B_1(K)$-module structure, with $y$ acting, from the right, as the formal derivative $\frac{d}{dx}$.
    In this way, $B_1(K)$ can be interpreted as the ring of linear differential operators (with rational functions as coefficients) on $K(x)$.

    From the fact that $B_1(K)$ is similarity factorial, one obtains results on the uniqueness of factorizations of homogeneous linear differential equations, as in \cite{landau02,loewy03}.

  \item The ring of Hurwitz quaternions, $\cH$, is Euclidean with respect to the reduced norm.
    This leads to an easy proof of Lagrange's Four-Square Theorem, in the same way that the ring of Gaussian integers $\bZ[i]$ can be used to obtain an easy proof of the Sum of Two Squares Theorem (see \cite[Theorem 26.6]{reiner75}) .
  \end{enumerate}
\end{example}

Free associative algebras in more than one indeterminate over a field are not PIDs and hence not Euclidean.
However, in the 1960s, P.\,M.~Cohn and Bergman developed the more general notion of an \emph{($n$-term) weak algorithm} (see \cite{cohn06}), which can be used to prove that a ring is an atomic $n$-fir.
We recall the definition, following \cite[Chapter~2]{cohn06}.

A \emph{filtration} on a ring $R$ is a function $v\colon R \to \bN_0 \cup \{-\infty\}$ satisfying the following conditions:
\begin{enumerate}
  \item For $a \in R$, $v(a) = -\infty$ if and only if $a = 0$.
  \item $v(a-b) \le \max\{v(a),v(b)\}$ for all $a$,~$b \in R$.
  \item $v(ab) \le v(a) + v(b)$ for all $a$,~$b \in R$.
  \item $v(1) = 0$.
\end{enumerate}

Equivalently, a filtration is defined by a family $\{0\}=R_{-\infty} \subset R_0 \subset R_1 \subset R_2 \subset \ldots$ of additive subgroups of $R$ such that $R = \bigcup_{i \in \bN_0 \cup \{-\infty\}} R_i$, for all $i$,~$j \in \bN_0 \cup \{-\infty\}$ it holds that $R_iR_j \subset R_{i+j}$, and $1 \in R_0$.
The equivalence of the two definitions is seen by setting $R_i = \{\, a \in R \mid v(a) \le i \,\}$, respectively, in the other direction, by setting $v(a) = \min\{\, i \in \bN_0 \cup \{-\infty\} \mid a \in R_i \,\}$.

Let $R$ be a ring with filtration $v$.
A family $(a_i)_{i \in I}$ in $R$ with index set $I$ is \emph{right $v$-dependent} if either $a_i=0$ for some $i \in I$, or there exist $b_i \in R$, almost all zero, such that
\[
v\Big(\sum_{i \in I} a_i b_i\Big) \;<\; \max_{i \in I}\; v(a_i) + v(b_i).
\]
If $a \in R$ and $(a_i)_{i \in I}$ is an family in $R$, then $a$ is \emph{right $v$-dependent on $(a_i)_{i\in I}$} if either $a=0$ or there exist $b_i \in R$, almost all zero, such that
\[
v\Big(a - \sum_{i \in I}a_i b_i\Big) < v(a) \quad\text{and}\quad v(a_i) + v(b_i) \le v(a) \text{ for all $i \in I$}.
\]

\begin{definition}
  For $n \in \bN$, a filtered ring $R$ satisfies the \emph{$n$-term weak algorithm} if, for any right $v$-dependent family $(a_{i})_{i \in [1,m]}$ of $m \le n$ elements with $v(a_1) \le v(a_2) \le \ldots \le v(a_m)$, there exists a $j \in [1,m]$ such that $a_j$ is right $v$-dependent on $(a_i)_{i\in [1,j-1]}$.
  $R$ satisfies the \emph{weak algorithm} if it satisfies the $n$-term weak algorithm for all $n \in \bN$.
\end{definition}
The asymmetry in the definition is only an apparent one.
A filtered ring $R$ satisfies the $n$-term weak algorithm with respect to the notion of right $v$-dependence if and only if the same holds true with respect to left $v$-de\-pen\-dence (see \cite[Proposition 2.4.1]{cohn06}).

If $R$ satisfies the $n$-term weak algorithm, then it also satisfies the $m$-term weak algorithm for $m < n$.
If $R$ satisfies the $1$-term weak algorithm, then $R$ is a domain and $v(ab) = v(a) + v(b)$ for all $a$,~$b \in R \setminus \{0\}$.
If moreover $R_0 \subset R^\times \cup \{0\}$, that is $R_0$ is a division ring, then $v$ induces a length function on $R^\bullet$.
In this case, $R$ is a BF-domain.
If $R$ satisfies the $n$-term weak algorithm for $n \ge 2$, then $R$ is a domain with $R_0 \subset R^\times \cup \{0\}$ a division ring.

Of particular interest is the $2$-term weak algorithm.
Explicitly, it says that for two elements $a$,~$b \in R$ which are right $v$-dependent, if $b \ne 0$ and $v(b) \le v(a)$, then there exists $c \in R$ such that $v(a - bc) < v(a)$.
Comparing with Equation~\eqref{e-euclidean}, we see that the existence of a $2$-term weak algorithm implies that a Euclidean division algorithm holds for elements $a$ and $b$ which are right $v$-dependent.
\begin{theorem}[{\cite[Proposition 2.4.8]{cohn06}, \cite[Proposition 2.2.7]{cohn85}}]
  Let $R$ be a filtered ring with $n$-term weak algorithm, where $n \ge 2$.
  Then $R$ is an $n$-fir and satisfies the ACC on $n$-generated left, respectively right, ideals.
  In particular, $R$ is similarity factorial.
\end{theorem}
We also note in passing that if $R$ is a filtered ring with weak algorithm then $R$ is not only a semifir but even a fir (see \cite[Theorem 2.4.6]{cohn06}).

\begin{example}
  A standard example shows that a right Euclidean domain need not be a left PID.
  Let $K$ be a field, and let $\sigma$ be the endomorphism of the rational function field $K(x)$ given by $\sigma(x) = x^2$ and $\sigma|_K = \id_K$.
  Then the skew polynomial ring $R = K(x)[y;\sigma]$ is right Euclidean, but does not even have finite uniform dimension as a left module over itself, as it contains an infinite direct sum of left ideals (see \cite[Example~1.2.11(ii)]{mcconnell-robson01}).
  However, since $R$ is right Euclidean, it has a $2$-term weak algorithm.
  Hence $R$ is an atomic $2$-fir and in particular similarity factorial.

  The notions of $n$-fir, similarity factoriality, and [$n$-term] weak algorithm are symmetric, while being a right PID and being right Euclidean are non-symmetric concepts.
\end{example}

Before we can state one of the main theorems on the existence of a weak algorithm, we have to recall $A$-rings (for a ring $A$), tensor $A$-rings, and coproducts of $A$-rings.
Let $A$ be a ring.
An \emph{$A$-ring} is a ring $R$ together with a ring homomorphism $A \to R$.
If $V$ is an $A$-bimodule, we set $V^{\otimes 0}=A$ and inductively $V^{\otimes n} = V^{\otimes (n-1)} \otimes_A V$ for all $n \in \bN$.
The \emph{tensor $A$-ring} $A[V]$ is defined as $A[V] = \bigoplus_{n \in \bN_0} V^{\otimes n}$, with multiplication induced by the natural isomorphisms $V^{\otimes m} \otimes_A V^{\otimes n} \to V^{\otimes (m+n)}$.
If $V$ is a free right $A$-module with basis $X$, then the free monoid $X^*$ generated by $X$ is a basis of the right $A$-module $A[V]$.
In this case, every $f \in A[V]$ has a unique representation of the form
\begin{equation}\label{e-freering}
f=\sum_{x \in X^*} x a_x \qquad\text{with $a_x \in A$, almost all zero}.
\end{equation}
Note however that elements of $A$ need not commute with elements from $X$.

If $V$ is a free right $A$-module with basis $X$, and a bimodule structure is defined on $V$ by means of $\lambda x = x \lambda$ for all $\lambda \in A$ and $x \in X$, then $A\langle X \rangle = A[V]$ is the \emph{free $A$-ring} on $X$.
By the choice of bimodule structure, elements from $A$ commute with elements from $X$ in $A\langle X \rangle$.
If $R$ and $S$ are $A$-rings, the coproduct $R \mathbin{*_{\!A}} S$ in the category of $A$-rings is the pushout of the homomorphisms $A \to R$ and $A \to S$ in the category of rings.

If $D$ is a division ring, $V$ is a $D$-bimodule, and $R$ and $S$ are filtered $D$-rings with $R_0\cong S_0 \cong D$, then $D[V]$ as well as $R *_{D} S$ are naturally filtered.
If $X$ is a set, one defines the \emph{free $R$-ring $R_D\langle X \rangle$} on the $D$-centralizing indeterminates $X$ as $R_D\langle X \rangle = R \mathbin{*_{\!D}} D\langle X \rangle$.
In $R_D\langle X \rangle$, elements of $D$ commute with elements of $X$.

\begin{theorem}[{\cite[Chapter 2.5]{cohn06}}]
  Let $D$ be a division ring.
  \begin{enumerate}
  \item Let $V$ be a $D$-bimodule.
        Then the tensor $D$-ring $D[V]$ satisfies the weak algorithm relative to the natural filtration.
  \item Let $R$, $S$ be $D$-rings with weak algorithm, where $R_{0}\cong S_{0} \cong D$.
    Then the coproduct $R \mathbin{*_D} S$ in the category of $D$-rings satisfies the weak algorithm relative to the natural filtration.
  \item Let $R$ be a ring with weak algorithm and $R_0 \cong D$.
    For any set $X$, the free $R$-ring $R_D\langle X \rangle=R \mathbin{*_D} D\langle X \rangle$ on $D$-centralizing indeterminates $X$ satisfies the weak algorithm relative to the natural filtration.
  \end{enumerate}
  In particular, these rings are firs and hence similarity factorial.
\end{theorem}
\begin{corollary}
  If $K$ is a field and $X$ is a set of noncommuting indeterminates, then the free associative $K$-algebra $K\langle X \rangle$ satisfies the weak algorithm.
  In particular, $K\langle X \rangle$ is a fir and hence similarity factorial.
\end{corollary}

In a similar fashion, the \emph{inverse weak algorithm} can be used to show that power series rings in any number of noncommuting indeterminates are similarity factorial (see \cite{cohn62} or \cite[Chapter~2.9]{cohn06}).
A \emph{transfinite weak algorithm} can be used to prove that certain semigroup algebras are right firs (see \cite[Chapter~2.10]{cohn06}).

For classical maximal orders in central simple algebras over global fields, we have the following result on similarity factoriality.
\begin{theorem}[{\cite[Corollary~7.14]{baeth-smertnig15}}]
  Let $R$ be a classical maximal $\cO$-order over a holomorphy ring $\cO$ in a global field.
  Suppose that every stably free right $R$-ideal is free.
  Then the following statements are equivalent.
  \begin{equivenumerate}
    \item $R$ is similarity factorial.
    \item Every right $R$-ideal is principal.
    \item Every left $R$-ideal is principal.
    \item The ray class group $\cC_A(\cO)$ is trivial.
    \end{equivenumerate}
\end{theorem}

\subsubsection{Rigid Domains}
A domain $R$ is \emph{rigid} if $[aR,R]$ is a chain for all $a \in R^\bullet$.
Rigid domains and rigid similarity factorial domains have been characterized by P.\,M.~Cohn.
Recall that a nonzero ring $R$ is \emph{local} if $R/J(R)$ is a division ring.
Here, $J(R)$ is the Jacobson radical of $R$.
\begin{theorem}[{\cite[Theorem~3.3.7]{cohn06}}]
  A domain is rigid if and only if it is a $2$-fir and a local ring.
\end{theorem}
\begin{lemma}\label{l-rigid-ufd}
  For a domain $R$, the following statements are equivalent.
  \begin{equivenumerate}
    \item\label{l-rigid-ufd:atomic-rigid} $R$ is rigid and atomic.
    \item\label{l-rigid-ufd:ufd} $R$ is rigid and similarity factorial.
    \item\label{l-rigid-ufd:factorial} $R$ is \emph{rigidly factorial} in the sense of \cite{baeth-smertnig15}. That is, $\card{\sZ^*(a)}=1$ for all $a \in R^\bullet$.
    \item\label{l-rigid-ufd:2fir} $R$ is an atomic $2$-fir and a local ring.
  \end{equivenumerate}
\end{lemma}
\begin{proof}
  \ref*{l-rigid-ufd:atomic-rigid}$\;\Leftrightarrow\;$\ref*{l-rigid-ufd:ufd}$\;\Leftrightarrow\;$\ref*{l-rigid-ufd:factorial} is trivial.
  The non-trivial equivalence \ref*{l-rigid-ufd:atomic-rigid}$\;\Leftrightarrow\;$\ref*{l-rigid-ufd:2fir} follows from the previous theorem.
\end{proof}

Note that a factorial commutative domain is rigid if and only if it is a discrete valuation ring.
The extreme restrictiveness of rigid domains is what requires one to study notions of factoriality which are weaker than rigid factoriality, such as similarity factoriality, where some degree of refactoring is permitted.
However, interesting rings which satisfy the equivalent conditions of Lemma~\ref{l-rigid-ufd} do exist: power series rings in any number of noncommuting indeterminates over a division ring (see \cite[Theorems 2.9.8 and 3.3.2]{cohn06}).

\subsubsection{Distributive Factor Lattices}

Let $R$ be a domain.
Then $R$ is a $2$-fir if and only if the factor posets $[aR,R]$ for $a \in R^\bullet$ are sublattices of the lattice of principal right ideals.
Hence, for all $a \in R^\bullet$, the factor lattice $[aR,R]$ is modular.
For a commutative Bézout domain, in fact, the factor lattices are distributive, since the lattice of fractional principal ideals is a lattice ordered group.
In the noncommutative setting this is no longer true in general.

\begin{example}
  Let $\cH$ be the ring of Hurwitz quaternions.
  Then $\cH$ is a PID and hence, in particular, a $2$-fir.
  If $p \in \bP \setminus \{2\}$ is an odd prime number, then $\cH/p\cH \cong M_2(\bF_p)$.
  Thus, $[p\cH,\cH]$ is isomorphic to the lattice of right ideals of $M_2(\bF_p)$.
  The lattice of right ideals of $M_2(\bF_p)$ is in turn isomorphic to the lattice of $\bF_p$-subspaces of $\bF_p^2$.
  Hence, $[p\cH,\cH]$ is not distributive.
\end{example}

A domain $R$, which is a $K$-algebra over a field $K$, is an \emph{absolute domain} if $R \otimes_K L$ is a domain for all algebraic field extensions $L$ of $K$.
If $R$ is moreover a 2-fir and $K(x)$ denotes the rational function field over $K$, the ring $R$ is a \emph{persistent $2$-fir} if $R \otimes_K K(x)$ is again a $2$-fir.
For instance, the free associative $K$-algebra $K\langle X \rangle$ on a set of indeterminates $X$ is an absolute domain and a persistent $2$-fir.

\begin{theorem}[{\cite[Theorem 4.3.3]{cohn06}}]
  Let $K$ be a field and let $R$ be a $K$-algebra that is an absolute domain and a persistent $2$-fir.
  Then the factor lattice $[aR,R]$ is distributive for all $a \in R^\bullet$.
\end{theorem}

There is a duality between the category of finite distributive lattices and the category of finite partially ordered sets.
It is given (in both directions), by mapping a distributive lattice $X$, respectively a partially ordered set $X$, to $\Hom(X,\{0, 1\})$ (see \cite[Chapter 4.4]{cohn06}).
Here $\{0,1\}$ is to be considered as two-element distributive lattice, respectively partially ordered set, with $0 < 1$.

Under this duality, the distributive lattices that appear as factor lattices in a factorial commutative domain correspond to disjoint unions of finite chains.
In contrast, in noncommutative similarity factorial domains, we have the following.
(This seems to go back to Bergman and P.\,M.~Cohn.)
\begin{theorem}[{\cite[Theorem 4.5.2]{cohn06}}]
  Let $K$ be a field and $R=K\langle x_1,\ldots,x_n \rangle$ with $n \ge 2$ a free associative algebra.
  Let $L$ be a finite distributive lattice.
  Then there exists $a \in R^\bullet$ with $[aR,R] \cong L$.
\end{theorem}

On the other hand, if $R$ is a PID, we have the following.
\begin{theorem}[{\cite[Theorem 4.2.8]{cohn06}}]
  Let $R$ be a PID.
  Then every factor lattice $[aR,R]$ for $a \in R^\bullet$ is distributive if and only if every element of $R^\bullet$ is normal.
\end{theorem}
Thus, every left (or right) ideal $I$ of $R$ is already an ideal of $R$, and $I=aR=Ra$ for a normal element $a \in R$.

\subsubsection{Comaximal Transposition/Metacommutation}

In an atomic $2$-fir $R$, it follows from the usual inductive proof of the Jordan-Hölder Theorem that every rigid factorization of an element can be transformed into any other rigid factorization of the same element by successively replacing two consecutive atoms by two new ones.
Using the arithmetical invariants that will be introduced in Section~\ref{ss-arithmetical-invariants} for the study of non-unique factorizations, this means $\sc^*(R^\bullet) \le 2$.

To understand factorizations in such rings in more detail, the following question is of central importance: Given two atoms $u$, $v \in R^\bullet$, what can be said about atoms $u'$, $v' \in R^\bullet$ such that $uv=v'u'$?
Such a relation is referred to as \emph{(comaximal) transposition} in the context of $2$-firs when $uR \ne v'R$, that is $uR+v'R=R$ (see \cite[Chapters 3.2 and 3.5]{cohn06}).
In \cite{conway-smith03}, in the context of the ring of Hurwitz quaternions, this problem is referred to as \emph{metacommutation} when $\nr(u)$ and $\nr(v)$ are coprime.

\begin{example}
Let $R$ be a classical maximal $\cO$-order in which every right [left] $R$-ideal is principal.
Consider two atoms $u$ and $v$ of $R$.
Suppose first $\nr(u) \nsimeq \nr(v)$.
Then there exist atoms $u'$, $v' \in R$ such that $uv=v'u'$, $\nr(u)\simeq\nr(u')$ and $\nr(v) \simeq \nr(v')$.
Moreover, $\rf{v',u'}$ is uniquely determined.
That is, if $u$ and $v$ have coprime reduced norms, then there is a unique (up to units) way of refactoring $uv$ such that the order of reduced norm is exchanged.

If $\nr(u) \simeq \nr(v)$, then the situation is more complicated.
The rigid factorization $\rf{u,v}$ can be the unique factorization of $uv$, or there can be many different rigid factorizations.
For instance, consider the ring $R=M_2(\bZ)$ and let $p \in \bP$ be a prime number.
Then $\begin{psmallmatrix} p^2 & 0 \\ 0 & 1 \end{psmallmatrix}$ has a unique rigid factorization, namely $\rf{\begin{psmallmatrix} p & 0 \\ 0 & 1 \end{psmallmatrix}, \begin{psmallmatrix} p & 0 \\ 0 & 1 \end{psmallmatrix}}$.
However, $\begin{psmallmatrix} p & 0 \\ 0 & p \end{psmallmatrix}$ has $p+1$ distinct rigid factorizations, given by $\rf{\begin{psmallmatrix} 1 & 0 \\ 0 & p \end{psmallmatrix}, \begin{psmallmatrix} p & 0 \\ 0 & 1 \end{psmallmatrix}}$ and $\rf{\begin{psmallmatrix} p & x \\ 0 & 1 \end{psmallmatrix}, \begin{psmallmatrix} 1 & -x \\ 0 & p \end{psmallmatrix}}$ with $x \in [0,p-1]$.
\end{example}

H.~Cohn and Kumar have studied the comaximal transposition (metacommutation) of atoms with coprime norm in the Hurwitz quaternions in detail.

\begin{theorem}[{\cite{cohn-kumar15}}]
  Let $\cH$ be the ring of Hurwitz quaternions, and let $p \ne q \in \bP$ be prime numbers.
  Let $v \in \cA(\cH)$ be an atom of reduced norm $q$, and let $\cA_p$ denote the set of left associativity classes of atoms of reduced norm $p$.
  Metacommutation with $v$ induces a permutation $\pi$ of $\cA_p$:
  If $\cH^\times u \in \cA_p$, there exist atoms $u'$ and $v'$ with $\nr(v')=q$, $\nr(u')=p$ and $uv=v'u'$, with the left associativity class $\cH^\times u'$ of $u'$ uniquely determined by $\cH^\times u$. Then $\pi(\cH^\times u)= \cH^\times u'$.

  \begin{enumerate}
  \item The sign of $\pi$ is the quadratic character $\big( \frac{q}{p} \big)$ of $q$ modulo $p$.
  \item If $p=2$ or $u \equiv n \mod p\cH$ for some $n \in \bZ$, then $\pi=\id_{\cA_p}$.
    Otherwise, $\pi$ has $1+ \Big(\frac{\tr(v)^2-q}{p}\Big)$ fixed points.
  \end{enumerate}
\end{theorem}

\subsubsection{Polynomial Rings}

If $D$ is a division ring, $\sigma$ is an injective endomorphism of $D$, and $\delta$ is a $\sigma$-derivation, we have already noted that the skew polynomial ring $D[x;\sigma,\delta]$ is a right Euclidean domain, and hence similarity factorial.
If $R$ is a factorial commutative domain, then the polynomial ring $R[x]$ is factorial as well.
This follows either from Gauss's lemma or from Nagata's Theorem.
The following two striking examples due to Beauregard show that a similar result cannot hold in the noncommutative setting in general.

\begin{theorem}[{\cite{beauregard92}}]\label{e-beau-hurwitz}
  Let $\cH$ denote the ring of Hurwitz quaternions.
  Then the polynomial ring $\cH[x]$ is not half-factorial.
  Explicitly, with atoms $a=1-i+k$, $f=a x^2 + (2+2i)x + (-1+i-2k)$, and $h=\frac{1}{2}(1-i+j+k)x^2+ (1+i)x + (-1+i)$, one has
  \[
  f \overline f = a \overline a h \overline h.
  \]
\end{theorem}

\begin{theorem}[{\cite{beauregard93}}]\label{e-beau-rational}
  Let $\bH_{\bQ}$ denote the Hamilton quaternion algebra with coefficients in $\bQ$.
  Then $\bH[x,y]$ is not half-factorial.
  Explicitly, with
  \[
  f=(x^2y^2-1) + (x^2 - y^2)i + 2xy j,
  \]
  one has
  \[
  f\overline f = (x^2 + i)(x^2 - i)(y^2 + i)(y^2 - i),
  \]
  with all stated factors being atoms.
\end{theorem}

Note that this is quite independent of the precise definition of factoriality we are using.
In particular, the second result implies that as long as we expect a factorial domain to be at least half-factorial and that division rings are (trivially) factorial domains, then it cannot be that polynomial rings over factorial domains are again always factorial domains.

\subsubsection{Weaker Forms of Similarity and Nagata's Theorem}

A basic form of Nagata's theorem in the commutative setting is the following:
Let $R$ be a commutative domain, and $S \subset R$ a multiplicative subset generated by prime elements.
Then, if $S^{-1}R$ is factorial, so is $R$.
In this way, one obtains that $\bZ[x]$ is factorial from the fact that $\bQ[x]$ is factorial.

A similar result cannot hold for similarity factoriality, as the following example from \cite{cohn69a} shows.
In $\bZ\langle x,y \rangle$, we have
\[
xyx + 2x = x(yx+2) = (xy+2)x.
\]
However, $yx+2$ is not similar to $xy+2$ in $\bZ\langle x, y\rangle$, as can be verified by a direct computation.

This provides a motivation to study weaker forms of equivalence relations on atoms than that of similarity.
Two elements $a$,~$b$ in a domain $R$ are called \emph{(right) subsimilar}, if there exist injective module homomorphisms $R/aR \hookrightarrow R/bR$ and $R/bR \hookrightarrow R/aR$.
Brungs, in \cite{brungs69}, studied domains in which factorizations are unique up to permutation and subsimilarity of atoms.

\begin{definition}
A domain $R$ is \emph{subsimilarity factorial} (or a \emph{subsimilarity-UFD}) if $R$ is atomic, and factorizations of elements are unique up to order and subsimilarity of the atoms.
\end{definition}

Brungs proved a form of Nagata's Theorem using this notion and a, in general somewhat complicated, concept of prime elements (see \cite[Satz 7]{brungs69}).
In turn, he obtained the following.
\begin{theorem}[{\cite[Satz 8]{brungs69}}]
  Let $R$ be a commutative domain and $X$ a set of noncommuting indeterminates.
  Then the free associative algebra $R\langle X \rangle$ is subsimilarity factorial if and only if $R$ is factorial.
\end{theorem}
In the same paper, Brungs showed that skew power series rings over right PIDs are right LCM domains.
He used this to construct an atomic right LCM domain which is not half-factorial (see Example~\ref{e-atomic-right-lcm-not-hf} below).

Motivated by Brungs' work, and with the goal of obtaining a variant of Nagata's Theorem with a simpler notion of prime elements than the one Brungs was using, P.\,M.~Cohn, in \cite{cohn69a}, introduced the notion of \emph{(right) monosimilarity}.
Let $R$ be a ring, and call an element $a \in R$ \emph{regular} if all divisors of $a$ are non-zero-divisors.
A right $R$-module is \emph{strictly cyclic} if it is isomorphic to $R/aR$ for a regular element $a \in R$.
The category $\cC_R$ of strictly cyclic modules is the full subcategory of the category of right $R$-modules with objects the strictly cyclic right $R$-modules.
If $R$ is a $2$-fir, then $\cC_R$ is an abelian category.

Two regular elements $a$,~$b \in R$ are called \emph{(right) monosimilar} if there exist monomorphisms $R/aR \to R/bR$ and $R/bR \to R/aR$ in $\cC_R$.
In general, this is a weaker notion than subsimilarity.
Indeed, if $R$ is a domain, then a homomorphism $f$ of strictly cyclic modules is a monomorphism in $\cC_R$ if and only if its kernel (as homomorphism of $R$-modules) is torsionfree.
Within the class of $2$-firs, the notions of subsimilarity and monosimilarity are equivalent to similarity.

In \cite{cohn73}, P.\,M.~Cohn gives a set of axioms for an equivalence relation on elements that is sufficient to obtain Nagata's theorem.
These axioms are satisfied by the (right) monosimilarity relation, but in general not by the similarity relation.
The main obstacle in the case of the similarity relation is that $a$ and $b$ being similar in $S^{-1}R$ does not imply that $a$ and $b$ are similar in $R$.

\subsubsection{Stronger Forms of Similarity}

A ring $R$ is \emph{permutably factorial} if $R^\bullet$ is atomic and factorizations in $R^\bullet$ are unique up to order and two-sided associativity of the atoms.
This was studied in \cite{baeth-smertnig15}.
It is a rather strong requirement, but there are results for $R=T_n(D)$, the ring of $n\times n$ upper triangular matrices over an atomic commutative domain $D$, and $R=M_n(D)$ when $D$ is commutative.
See Sections~\ref{sssec:matrix} and \ref{sssec:triangular} below.

In \cite{beauregard94b}, Beauregard studied right UFDs.
A domain $R$ is a \emph{right unique factorization domain (right UFD)} if it is atomic, and factorizations are unique up to order and right associativity of the atoms.
Note that Example~\ref{e-sim-ufd}\ref*{e-sim-ufd:maxord} implies that there exist PIDs which are not right UFDs.
Beauregard gives an example of a right UFD which is not a left UFD.
In particular, while any right or left UFD is permutably factorial, the converse is not true.
(This can also be seen by looking at $M_n(R)$ for $R$ a commutative PID, $n \ge 2$, and using the Smith Normal Form.)

\subsubsection{LCM Domains and Projectivity Factoriality}

LCM domains and factorizations of elements therein were investigated by Beauregard in a series of papers (see \cite{beauregard71, beauregard74, beauregard77, beauregard80, beauregard94, beauregard95}).
A domain $R$ is a \emph{right LCM domain} if $aR \cap bR$ is principal for all $a$,~$b \in R$.
A \emph{left LCM domain} is defined analogously, and an \emph{LCM domain} is a domain which is both, a right and a left LCM domain.
By the characterization in Theorem~\ref{t-2fir}, any $2$-fir is an LCM domain.

If $R$ is an LCM domain and $a \in R^\bullet$, then the poset $[aR,R]$ is a lattice with respect to the partial order induced by set inclusion (see \cite[Lemma~1]{beauregard71}).
However, $[aR,R]$ need not be a sublattice of the lattice of all right ideals of $R$, that is, $bR+cR$ need not be principal for $bR$,~$cR \in [aR,R]$.

A commutative domain is an atomic LCM domain if and only if it is factorial.
Unfortunately, if $R$ is an atomic right LCM domain, $R$ need not even be half-factorial, as the following example shows.
\begin{example}[{\cite{brungs69} or \cite[Remark 3.9]{beauregard77}}] \label{e-atomic-right-lcm-not-hf}
  Let $R=K[x]$ be the polynomial ring over a field $K$.
  Let $\sigma\colon R \to R$ be the monomorphism with $\sigma|_{K}=\id_K$ and $\sigma(x)=x^2$.
  The skew power series ring $S=R\llbracket y;\sigma\rrbracket$, consisting of elements of the form $\sum_{n=0}^\infty y^n a_n$ with $a_n \in R$, and with multiplication given by $ay=y\sigma(a)$ for $a \in R$, is a right LCM domain by \cite[Satz 9]{brungs69}.
  The equality $xy=yx^2$ shows that $S$ is not half-factorial.
\end{example}

However, under an additional condition we do obtain unique factorization in a sense.
For $a$,~$b \in R^\bullet$, denote by $[a,b]_r$ a least common right multiple (LCRM), that is, a generator of $aR \cap bR$, and by $(a,b)_l$ a greatest common left divisor (GCLD), that is, a generator of the least principal ideal containing $aR + bR$.
Note that $[a,b]_r$ and $(a,b)_l$ are only defined up to right associativity.
A right LCM domain is called \emph{modular} if, for all $a$, $b$,~$c \in R^\bullet$,
\begin{equation*}
  [a,bc]_r=[a,b]_r \text{ and } (a,bc)_l = (a,b)_l \text{ implies $c \in R^\times$.}
\end{equation*}
If $R$ is an LCM domain, the condition is equivalent to the lattice $[aR,R]$ being modular.
Thus, any $2$-fir is a modular LCM domain.
However, the converse is not true.
Any factorial commutative domain which is not a PID is a counterexample.

Let $R$ be a domain. Beauregard calls two elements $a$,~$a' \in R^\bullet$ \emph{transposed}, and writes $a \tr a'$, if there exists $b \in R^\bullet$ such that
\[
[a,b]_r=ba' \quad\text{and}\quad (a,b)_l=1.
\]
If this is the case, there exists $b' \in R^\bullet$ such that $ba'=ab'$ and $b \tr b'$.
If $R$ is an LCM domain, then $a \tr a'$ if and only if the interval $[aR,R]$ is down-perspective to $[ba'R,bR]$ in the lattice $[ba'R,R]$.
If $R$ is a $2$-fir, then $a$ and $a'$ are transposed if and only if they are similar.
The elements $a$ and $a'$ are \emph{projective} if there exist $a=a_0$, $a_1$, $\ldots\,$,~$a_n=a'$ such that, for each $i \in [1,n]$, either $a_{i-1} \tr a_i$ or $a_i \tr a_{i-1}$.
\begin{definition}
  A domain $R$ is \emph{projectivity factorial} (or a \emph{projectivity-UFD}) if $R$ is atomic, and factorizations of elements are unique up to order and projectivity of the atoms.
\end{definition}
\begin{theorem}[{\cite[Theorem 2]{beauregard71}}]
  If $R$ is an atomic modular right LCM domain, then $R$ is projectivity factorial.
\end{theorem}

In \cite{beauregard95}, the condition of modularity has been weakened to the \emph{right atomic multiple property (RAMP)}.
A domain satisfies the RAMP if, for elements $a$,~$b \in R$ with $a$ an atom and $aR \cap bR \ne 0$, there exist $a'$,~$b' \in R$ with $a'$ an atom such that $ab'=ba'$.
One can check that, for an LCM domain, the RAMP is equivalent to the lattice $[aR,R]$ being lower semimodular for all $a \in R^\bullet$.
An atomic LCM domain is modular if and only if it satisfies both, the RAMP and LAMP, which is defined symmetrically (see \cite[Theorem 3]{beauregard95}).

Beauregard shows that, in a right LCM domain $R$, the RAMP is equivalent to the following condition: If $a$,~$a' \in R^\bullet$ such that $a$ is an atom, and $a \tr a'$, then $a'$ is also an atom (see \cite[Proposition~2]{beauregard95}).
He obtains the following generalization of the previous theorem.
\begin{theorem}[{\cite[Theorem~1]{beauregard95}}]
  If $R$ is an atomic right LCM domain satisfying the RAMP, then $R$ is projectivity factorial.
\end{theorem}
If $R$ is an atomic LCM domain, this theorem (as well as the previous one) can be deduced from the Jordan-Hölder theorem for semimodular lattices (see, for instance, \cite{graetzer-nation10}).
To do so, note the following: If $a$,~$a' \in R^\bullet$ and there exists $b \in R^\bullet$ such that interval $[aR,R]$ is projective to $[ba'R,bR]$ in the lattice $[ba'R,R]$, then the elements $a$ and $a'$ are projective.

Beauregard has also obtained a form of Nagata's Theorem for modular right LCM domains (see \cite{beauregard77}).
He has moreover shown that an atomic LCM domain with conjugation is already modular (see \cite[Theorem 4]{beauregard95}).
In \cite[Example~3]{beauregard95} he gives an example of an LCM domain which satisfies neither the RAMP nor the LAMP, and hence, in particular, does not have modular factor lattices.

Skew polynomial rings over total valuation rings provide another source of LCM domains.
A subring $V$ of a division ring $D$ is called a \emph{total valuation ring} if $x \in V$ or $x^{-1} \in V$ for each $x \in D^\bullet$.
\begin{theorem}[{\cite{marubayashi10}}]
Let $V$ be a total valuation ring, let $\sigma$ be an automorphism of $V$ and let $\delta$ be a $\sigma$-derivation on $V$ such that $\delta(J(V))\subset J(V)$, where $J(V)$ denotes the Jacobson radical of $V$.
Then $V[x;\sigma,\delta]$ is an LCM domain.
\end{theorem}

\subsection{A Different Notion of UFRs and UFDs}
\label{ssec:chatters-ufds}

A commutative domain is factorial if and only if every nonzero prime ideal contains a prime element.
Based on this characterization, Chatters introduced Noetherian unique factorization domains (Noetherian UFDs) in \cite{chatters84}.
Noetherian UFDs were generalized to Noetherian unique factorization rings (Noetherian UFRs) by Chatters and Jordan in \cite{chatters-jordan86}.

Noetherian UFDs and UFRs, and generalizations thereof, have received quite a bit of attention and found many applications (e.g., \cite{gilchrist-smith84,brown85,lebruyn86,abbasi-kobayashi-marubayashi-ueda91,chatters-clark91,chatters-gilchrist-wilson92,chatters95,jespers-wang01,jespers-wang02,launois-lenagan-rigal06,goodearl-yakimov,goodearl-yakimov14}).
A large number of examples of Noetherian UFDs have been exhibited in the form of universal enveloping algebras of finite-dimensional solvable complex Lie algebras as well as various semigroup algebras.
Moreover, Noetherian UFRs are preserved under the formation of polynomial rings in commuting indeterminates.

UFRs, respectively UFDs, which need not be Noetherian, were introduced by Chatters, Jordan, and Gilchrist in \cite{chatters-gilchrist-wilson92}.
Many Noetherian Krull orders turned out not to be Noetherian UFRs in the sense of \cite{chatters-gilchrist-wilson92}, despite having a factorization behavior similar to Noetherian UFRs.
This was the motivation for Abbasi, Kobayashi, Marubayashi, and Ueda to introduce the notion of a ($\sigma$-)UFR in \cite{abbasi-kobayashi-marubayashi-ueda91}, which provides another generalization of Noetherian UFRs.

Let $R$ be a prime ring.
An element $n \in R$ is \emph{normal} provided that $Rn=nR$.
We denote the subsemigroup of all normal elements of $R$ by $N(R)$.
Since $R$ is a prime ring, $N(R)^\bullet= N(R)\setminus\{0\}$ is a subset of the non-zero-divisors of $R$.
An element $p \in R \setminus \{0\}$ is \emph{prime} if $p$ is normal and $pR$ is a prime ideal.
An element $p \in R \setminus \{0\}$ is \emph{completely prime} if $p$ is normal and $pR$ is a completely prime ideal, that is, $R/pR$ is a domain.
If $R$ is Noetherian and $p \in R$ is a prime element, the principal ideal theorem (see \cite[Theorem 4.1.11]{mcconnell-robson01}) implies that $pR$ has height one.
\begin{definition}[{\cite{chatters-gilchrist-wilson92}}]
  Let $R$ be a ring.
  \begin{enumerate}
    \item $R$ is a \emph{unique factorization ring}, short \emph{UFR}, (in the sense of \cite{chatters-gilchrist-wilson92}) if it is a prime ring and every nonzero prime ideal of $R$ contains a prime element.
    \item $R$ is a \emph{unique factorization domain}, short \emph{UFD}, (in the sense of \cite{chatters-gilchrist-wilson92}) if it is a domain and every nonzero prime ideal of $R$ contains a completely prime element.
  \end{enumerate}
\end{definition}
Some remarks on this definition and its relation to the definitions of Noetherian UFRs and Noetherian UFDs in \cite{chatters-jordan86}, respectively \cite{chatters84}, are in order.
\begin{remark}
  \begin{enumerate}
  \item In \cite{chatters-jordan86}, \emph{Noetherian UFRs} were defined.
    A ring $R$ is a \emph{Noetherian UFR} in the sense of \cite{chatters-jordan86} if and only if it is a UFR in the sense of \cite{chatters-gilchrist-wilson92} and Noetherian.

  \item We will call a domain $R$ a \emph{Noetherian UFD} if it is a UFD and Noetherian.
    Except in Theorem~\ref{t-chatters-n-ufd}, we will not use the original definition of a \emph{Noetherian UFD} from \cite{chatters84}.
    A domain $R$ is a \emph{Noetherian UFD} in the sense of \cite{chatters84} if it contains at least one height one prime ideal and every height one prime ideal of $R$ is generated by a completely prime element.

    For a broad class of rings the two definitions of Noetherian UFDs agree.
    Suppose that $R$ is a Noetherian domain which is not simple.
    If every nonzero prime ideal of $R$ contains a height one prime ideal, then $R$ is a Noetherian UFD in the sense of \cite{chatters84} if and only if it is a UFD in the sense of \cite{chatters-gilchrist-wilson92}.
    If $R$ is a UFR or $R$ satisfies the descending chain condition (DCC) on prime ideals, then every nonzero prime ideal contains a height one prime ideal.
    In general, it is open whether every Noetherian ring satisfies the DCC on prime ideals (see \cite[Appendix, \S3]{goodearl-warfield04}).

  \item We warn the reader that a [Noetherian] UFR which is a domain need not be a [Noetherian] UFD: prime elements need not be completely prime.
    See Example~\ref{e-noetherian-ufr-not-ufd} below.
  \end{enumerate}
\end{remark}

From the point of view of factorization theory, UFRs and UFDs of this type are quite different from similarity factorial domains.
UFRs have the property that the subsemigroup $N(R)^\bullet$ of nonzero normal elements is a UF-monoid (see Theorem~\ref{t-akmu-ufr-ufm}).
However, if $R$ is a UFR, the prime elements of $N(R)^\bullet$ need not be atoms of $R^\bullet$.
If $R$ is a UFD, then prime elements of $N(R)^\bullet$ are indeed atoms in $R^\bullet$.
However, since they also need to be normal, this is in some sense quite a restrictive condition.
Nevertheless, many interesting examples of (Noetherian) UFRs and UFDs exist.
\begin{example}
  \begin{enumerate}
  \item Universal enveloping algebras of finite-dimensional solvable Lie algebras over $\bC$ are Noetherian UFDs (see \cite{chatters84}).
  \item Trace rings of generic matrix rings are Noetherian UFRs (see \cite{lebruyn86}).
  \item Let $R$ be a commutative ring and $G$ a polycyclic-by-finite group.
    It has been characterized when the group algebra $R[G]$ is a Noetherian UFR, respectively a Noetherian UFD.
    See \cite{brown85,chatters-clark91} and also \cite{chatters95}.
    There exist extensions of these results to semigroup algebras (see \cite{jespers-wang01,jespers-wang02}).
    Also see the book \cite{jespers-okninski07}.
  \item Certain iterated skew polynomial rings are Noetherian UFDs. This has been used to show that many quantum algebras are Noetherian UFDs. See \cite{launois-lenagan-rigal06}.
  \item Let $R$ be a Noetherian UFR.
    Then also $M_n(R)$ for $n \in \bN$ as well as $R[x]$ are Noetherian UFRs.
    It has been studied when $R[x;\sigma]$, with $\sigma$ an automorphism, and $R[x;\delta]$ are UFRs (see \cite{chatters-jordan86}).
  \end{enumerate}
\end{example}
We refer to the survey \cite{akalan-marubayashi} for more comprehensive results on the behavior of UFRs and UFDs under ring-theoretic constructions.

In \cite{abbasi-kobayashi-marubayashi-ueda91}, a generalization of Noetherian UFRs is introduced (even more generally, when $R$ is a ring and $\sigma$ is an automorphism of $R$, the notion of $\sigma$-UFR is defined).
Let $R$ be a prime Goldie ring and let $Q$ be its simple Artinian quotient ring.
For $X \subset R$, let $(R:X)_l = \{\, q \in Q \mid qX \subset R \,\}$ and $(R:X)_r = \{\, q \in Q \mid Xq \subset R\,\}$.
For a right $R$-ideal $I$, that is, a right ideal $I$ of $R$ containing a non-zero-divisor, let $I_v = (R:(R:X)_l)_r$, and for a left $R$-ideal $I$, let ${}_v I = (R:(R:X)_r)_l$.
A right [left] $R$-ideal $I$ is called \emph{divisorial} (or \emph{reflexive}) if $I = I_v$ [$I={}_v I$].
We refer to any of \cite{akalan-marubayashi,chamarie81,marubayashi-vanoystaeyen12} for the definition of right [left] $\tau$-$R$-ideals. (The terminology in \cite{chamarie81} is slightly different in that such right [left] ideals are called \emph{fermé}.)
Recall that any right [left] $\tau$-$R$-ideal is divisorial.
In particular, if $R$ satisfies the ACC on right [left] $\tau$-$R$-ideals, it also satisfies the ACC on divisorial right [left] $R$-ideals.

\begin{definition}[{\cite{abbasi-kobayashi-marubayashi-ueda91}}] \label{d-akmu}
  A prime Goldie ring $R$ is a \emph{UFR (in the sense of \cite{abbasi-kobayashi-marubayashi-ueda91})} if
  \begin{enumerate}
    \item $R$ is $\tau$-Noetherian, that is, it satisfies the ACC on right $\tau$-$R$-ideals as well as the ACC on left $\tau$-$R$-ideals.
    \item\label{d-akmu:ii} Every prime ideal $P$ of $R$ such that $P=P_v$ or $P={}_v P$ is principal.
  \end{enumerate}
\end{definition}
Equivalent characterizations, including one in terms of the factorizations of normal elements, can be found in \cite[Proposition 1.9]{abbasi-kobayashi-marubayashi-ueda91}.

\begin{theorem}[{\cite[Proposition 1.9]{abbasi-kobayashi-marubayashi-ueda91}}]
  Let $R$ be a $\tau$-Noetherian prime Goldie ring.
  Then the following statements are equivalent.
  \begin{equivenumerate}
    \item $R$ is a UFR in the sense of \cite{chatters-gilchrist-wilson92}.
    \item $R$ is a UFR in the sense of \cite{abbasi-kobayashi-marubayashi-ueda91} and the localization $(N(R)^\bullet)^{-1}R$ is a simple ring.
  \end{equivenumerate}
\end{theorem}

Following P.\,M.~Cohn, a cancellative normalizing semigroup $H$ is called a \emph{UF-monoid} if $H/H^\times$ is a free abelian monoid.
Equivalently, $H$ is a normalizing Krull monoid in the sense of \cite{geroldinger13} with trivial divisor class group.

\begin{theorem}[{\cite{abbasi-kobayashi-marubayashi-ueda91,chatters-gilchrist-wilson92}}] \label{t-akmu-ufr-ufm}
If $R$ is a UFR in the sense of \cite{abbasi-kobayashi-marubayashi-ueda91} or a UFR in the sense of \cite{chatters-gilchrist-wilson92}, then $N(R)^\bullet = N(R)\setminus\{0\}$ is a UF-monoid.
Explicitly, every nonzero normal element $a \in N(R)^\bullet$ can be written in the form
\[
a = \varepsilon p_1 \cdots p_n
\]
with $n \in \bN_0$, a unit $\varepsilon \in R^\times$, and prime elements $p_1$,~$\ldots\,$,~$p_n$ of $R$.
This representation is unique up to order and associativity of the prime elements.
\end{theorem}
\begin{remark}
  The unique factorization property for normal elements has been taken as the definition of another class of rings, studied by Jordan in \cite{jordan89}.
  Jordan studied \emph{Noetherian UFN-rings}, that is, Noetherian prime rings $R$ such that every nonzero ideal of $R$ contains a nonzero normal element and $N(R)^\bullet$ is a UF-monoid.
\end{remark}

Noetherian UFDs in the sense of \cite{chatters84} can be characterized in terms of factorizations of elements.
If $P$ is a prime ideal of a ring $R$, denote by $C(P) \subset R$ the set of all elements of $R$ whose images in $R/P$ are non-zero-divisors.

\begin{theorem}[{\cite{chatters84}}] \label{t-chatters-n-ufd}
  Let $R$ be a prime Noetherian ring.
  Set $C = \bigcap C(P)$, where the intersection is over all height one primes $P$ of $R$.
  The following statements are equivalent.
  \begin{equivenumerate}
    \item $R$ is a Noetherian UFD in the sense of \cite{chatters84}.
    \item\label{t-chatter-n-ufd:el} Every nonzero element $a \in R^\bullet$ is of the form $a=cp_1\cdots p_n$ for some $c \in C$ and completely prime elements $p_1$,~$\ldots\,$,~$p_n$ of $R$.
  \end{equivenumerate}
\end{theorem}

We note that property \ref*{t-chatter-n-ufd:el} of the previous theorem also holds for Noetherian UFDs in the sense of \cite{chatters-gilchrist-wilson92}.
If $C \subset R^\times$, then $R=N(R)$, and hence $R^\bullet$ is a UF-monoid.

In a ring $R$ which is a UFR, a prime element $p$ of $R$ is an atom of $N(R)^\bullet$ but need not be an atom in the (possibly larger) semigroup $R^\bullet$.
On the other hand, if $R$ is a UFD, the additional condition that $R/pR$ be a domain forces $p$ to be an atom.

\begin{example}[{\cite{chatters-jordan86}}] \label{e-noetherian-ufr-not-ufd}
Let $\bH_\bQ$ be the Hamilton quaternion algebra with coefficients in $\bQ$.
The ring $R=\bH_{\bQ}[x]$ is a Noetherian UFR and a domain, but $R$ is no UFD.
The element $x^2+1$ is central and generates a height one prime ideal, but $(x^2+1)R$ is not completely prime.
Thus, $R$ is not a UFD, even though it is Euclidean.
The element $x^2 + 1$ is an atom in $N(R)^\bullet$.
However, in $R^\bullet$, it factors as $x^2 + 1 = (x-i)(x+i)$.
\end{example}

Thus many interesting rings are UFRs but not UFDs.
This is especially true in the case of classical maximal orders in central simple algebras over global fields.
In this case, all but finitely many associativity classes of prime elements of $N(R)^\bullet$ are simply represented by the prime elements of the center of $R$.
We elaborate on this in the following example.

\begin{example}
  \begin{enumerate}
  \item \label{e-ufr-non-ufd:maxord}
    Let $\cO$ be a holomorphy ring in a global field $K$, and let $A$ be a central simple $K$-algebra with $\dim_K A=n^2 > 1$.
    Let $R$ be a classical maximal $\cO$-order.

    If $\fp$ is a prime ideal of $\cO$ such that $\fp$ is unramified in $R$ (i.e., $\fp R$ is a maximal ideal of $R$), then $\fp R$ is a height one prime ideal of $R$, and $R/\fp R \cong M_n(\cO/\fp)$.
    Thus, if $\fp=p\cO$ is principal, then $p$ is a prime element of $R$ which is not completely prime.
    Recall that at most finitely many prime ideals of $\cO$ are ramified in $R$.
    Thus, $R$ is not a UFD.
    However, $R$ is a Noetherian UFR if and only if the normalizing class group of $R$, that is, the group of all fractional $R$-ideals modulo the principal fractional $R$-ideals (generated by normalizing elements), is trivial.

  \item
    Elaborating on \ref*{e-ufr-non-ufd:maxord} in a specific example, the ring of Hurwitz quaternions $\cH$ is Euclidean and a Noetherian UFR, but not a UFD.
    The only completely prime element in $\cH$ (up to right associativity) is $1+i$.
    If $p$ is an odd prime number, then $p$ is a prime element of $\cH$ which is not completely prime, since $M_2(\cH/p\cH) \cong M_2(\bF_p)$.
    A complete set of representatives for associativity classes of prime elements of $\cH$ is given by $\{ 1 + i \} \,\cup\, \bP\setminus \{2\}$.
    If $p$ is an odd prime number, the $p+1$ maximal right $\cH$-ideals containing the maximal $\cH$-ideal $p\cH$ are principal and correspond to right associativity classes of atoms of reduced norm $p$.
    Thus $\card{\sZ_{\cH}^*(p)} = p+1$.
    However, all atoms of reduced norm $p$ are similar.
    As already observed, $\cH$ is similarity factorial.

  \item If $R$ is a commutative Dedekind domain with class group $G$, and $\exp(G)$ divides $r$, then $M_r(R)$ is a Noetherian UFR, but not a UFD if $r > 1$.
  \end{enumerate}
\end{example}

We say that a prime ring $R$ is \emph{bounded} if every right $R$-ideal and every left $R$-ideal contains a nonzero ideal of $R$.
Recall that every prime PI ring is bounded.
In \cite{gilchrist-smith84}, Gilchrist and Smith showed that every bounded Noetherian UFD which is not commutative is a PID.
This was later generalized to the following.
\begin{theorem}[{\cite{chatters-gilchrist-wilson92}}]
  Let $R$ be a UFD in the sense of \cite{chatters-gilchrist-wilson92}.
  Let $C = \bigcap C(P)$, where the intersection is over all height one prime ideals $P$ of $R$.
  If $C \subset R^\times$, then $R$ is duo.
  That is, every left or right ideal of $R$ is an ideal of $R$.
  Moreover, if $R$ is not commutative, then $R$ is a PID.
\end{theorem}
Hence, ``noncommutative UFDs are often PIDs,'' as the title of \cite{gilchrist-smith84} proclaims.

\section{Non-unique Factorizations}
\label{sec:non-unique-factorization}

We now come to non-unique factorizations.
We have already noted that a ring $R$ satisfying the ascending chain condition on principal left ideals and on principal right ideals is atomic.
In particular, this is true for any Noetherian ring.
Thus, we can consider rigid factorizations of elements in $R^\bullet$.
However, the conditions which are sufficient for various kinds of uniqueness of factorizations are much stricter.
Hence, a great many natural examples of rings have some sort of non-unique factorization behavior.

\subsection{Arithmetical Invariants} \label{ss-arithmetical-invariants}

The study of non-unique factorizations proceeds by defining suitable arithmetical invariants intended to capture various aspects of the non-uniqueness of factorizations.
The following invariants are defined in terms of lengths of factorizations, and have been investigated in commutative settings before.

\begin{definition}[Arithmetical invariants based on lengths] \label{d-len-inv}
  Let $H$ be a cancellative small category.
  \begin{enumerate}
    \item $\sL(a)=\{\, \length{z} \mid z \in \sZ^*(a) \,\}$ is the \emph{set of lengths} of $a \in H$.
    \item $\cL(H) = \{\, \sL(a) \mid a \in H \,\}$ is the \emph{system of sets of lengths} of $H$.
  \end{enumerate}
  Let $H$ be atomic.
  \begin{enumerate}
    \setcounter{enumi}{2}
    \item For $a \in H\setminus H^\times$,
      \[
      \rho(a) = \frac{\sup \sL(a)}{\min \sL(a)}
      \]
      is the \emph{elasticity} of $a$, and $\rho(a)=1$ for $a \in H^\times$.
    \item $\rho(H) = \sup \{\, \rho(a) \mid a \in H \,\}$ is the \emph{elasticity} of $H$.
    \item The invariants
      \[
      \rho_k(H) = \sup\big\{\, \sup \sL(a) \mid a \in H \text{ with }\min \sL(a) \le k \,\big\},
      \]
      for $k \in \bN_{\ge 2}$, are the \emph{refined elasticities} of $H$.
    \item For $k \in \bN_{\ge 2}$,
      \[
      \cU_k(H) = \bigcup_{\substack{L \in \cL(H) \\ k \in L}} L
      \]
      is the union of all sets of lengths containing $k$.

    \item If $a \in H$ with $k$,~$l \in \sL(a)$ and $[k,l] \cap \sL(a) = \{k,l\}$, then $l-k$ is a \emph{distance} of $a$.
    We write $\Delta(a)$ for the \emph{set of distances} of $a$.
    \item The \emph{set of distances} of $H$ is $\Delta(H) = \bigcup_{a \in H} \Delta(a)$.
  \end{enumerate}
\end{definition}

\begin{example}
  \begin{enumerate}
    \item
      Let $\bH_\bQ$ denote the Hamilton quaternion algebra with coefficients in $\bQ$, and let $\cH$ denote the ring of Hurwitz quaternions.
      Beauregard's results (Theorems~\ref{e-beau-hurwitz} and \ref{e-beau-rational}) imply $\rho_2(\cH[x]) \ge 4$ and $\rho_2(\bH_\bQ[x,y]) \ge 4$.
      Hence $\rho(\cH[x]) \ge 2$ and $\rho(\bH_\bQ[x,y]) \ge 2$.

    \item
      If $A_1(K)=K\langle x,y \mid xy-yx=1 \rangle=K[x][y;-\frac{d}{dx}]$ denotes the first Weyl algebra over a field $K$ of characteristic $0$, the example $x^2y=(1+xy)x$ of P.\,M.~Cohn shows $\rho_2(A_1(K)) \ge 3$, and hence $\rho(A_1(K)) \ge \frac{3}{2}$.
    \end{enumerate}
\end{example}

Recall that $H$ is \emph{half-factorial} if $\card{\sL(a)} = 1$ for all $a \in H$ (equivalently, $H$ is atomic, and $\Delta(H)=\emptyset$ or $\rho(H)=1$).
Since all the invariants introduced so far are defined in terms of sets of lengths, they are trivial if $H$ is half-factorial.

It is more difficult to make useful definitions for the more refined arithmetical invariants, such as catenary degrees, the $\omega$-invariant, and the tame degree, in a noncommutative setting.
In \cite{baeth-smertnig15}, a formal notion of distances between rigid factorizations was introduced.
This allows the definition and study of catenary degrees and monotone catenary degrees.

\begin{definition}[Distances]
  Let $H$ be a cancellative small category.
  A \emph{global distance on $H$} is a map $\sd \colon \sZ^*(H) \times \sZ^*(H) \to \bN_0$ satisfying the following properties.
  \begin{enumerate}[label=\textup{(\textbf{D\arabic*})},ref=\textup{(D\arabic*)}]
    \item\label{d:ref} $\sd(z,z) = 0$ for all $z \in \sZ^*(H)$.
    \item\label{d:sym} $\sd(z,z') = \sd(z',z)$ for all $z$,~$z' \in \sZ^*(H)$.
    \item\label{d:tri} $\sd(z,z') \le \sd(z,z'') + \sd(z'',z')$ for all $z$,~$z'$,~$z'' \in \sZ^*(H)$.
    \item\label{d:mul} For all $z$,~$z' \in \sZ^*(H)$ with $s(z)=s(z')$ and $x \in \sZ^*(H)$ with $t(x) = s(z)$ it holds that $\sd(x\rfop z, x \rfop z') = \sd(z,z')$, and for all $z$,~$z' \in \sZ^*(H)$ with $t(z)=t(z')$ and $y \in \sZ^*(H)$ with $s(y) = t(z)$ it holds that $\sd(z\rfop y, z'\rfop y) = \sd(z,z')$.
    \item\label{d:len} $\babs{\length{z} - \length{z'}} \le \sd(z,z') \le \max\big\{ \length{z}, \length{z'}, 1 \big\}$ for all $z$,~$z' \in \sZ^*(H)$.
  \end{enumerate}
  Let $L = \{\, (z, z') \in \sZ^*(H) \times \sZ^*(H) : \pi(z) = \pi(z') \,\}$.
  A \emph{distance on $H$} is a map $\sd \colon L \to \bN_0$ satisfying properties \ref*{d:ref}, \ref*{d:sym}, \ref*{d:tri}, \ref*{d:mul}, and \ref*{d:len} under the additional restrictions on $z$,~$z'$ and $z''$ that $\pi(z)=\pi(z')=\pi(z'')$.
\end{definition}

Let us revisit the notion of factoriality using distances as a tool.
We follow \cite[Section~3]{baeth-smertnig15}.
If $\sd$ is a distance on $H$, we can define a congruence relation $\sim_\sd$ on $\sZ^*(H)$ by $z \sim_\sd z'$ if and only if $\pi(z)=\pi(z')$ and $\sd(z,z')=0$.
That is, two factorizations are identified if they are factorizations of the same element and are at distance zero from each other.

\begin{definition}
Let $H$ be a cancellative small category, and let $\sd$ be a distance on $H$.
The quotient category $\sZ_\sd(H)= \sZ^*(H) / \sim_\sd$ is called the \emph{category of $\sd$-factorizations}.
The canonical homomorphism $\pi\colon\sZ^*(H) \to H$ induces a homomorphism $\pi_\sd\colon \sZ_\sd(H) \to H$.
For $a \in H$, we call $\sZ_\sd(a)=\pi_\sd^{-1}(a)$ the \emph{set of $\sd$-factorizations of $a$}.
We say that $H$ is \emph{$\sd$-factorial} if $\card{\sZ_d(a)} = 1$ for all $a \in H$.
\end{definition}

\begin{example}
  \begin{enumerate}
    \item We may define a so-called \emph{rigid distance} $\sd^*$.
      Informally speaking, $\sd^*(z,z')$ is the minimal number of replacements, deletions, and insertions of atoms necessary to pass from $z$ to $z'$.
      (The actual definition is more complicated to take into account the presence of units and the necessity to replace, delete, or insert longer factorizations than just atoms.)
      If $\sd^*(z,z') = 0$, then $z=z'$, and hence $\sZ_{\sd^*}(H) = \sZ^*(H)$.
      We say that $H$ is \emph{rigidly factorial} if it is $\sd^*$-factorial.

    \item \label{e-distances:equiv}
      Let $\sim$ be an equivalence relation on the set of atoms $\cA(H)$ such that $v = \varepsilon u \eta$ with $\varepsilon$,~$\eta \in H^\times$ implies $u \sim v$.
      Then, comparing atoms up to order and equivalence with respect to $\sim$ induces a global distance $\sd_{\sim}$ on $\sZ^*(H)$.
      Let $R$ be a domain, $H = R^\bullet$, and consider for the equivalence relation $\sim$ one of similarity, subsimilarity, monosimilarity, or projectivity.
      Then $R$ is $\sd_\sim$-factorial if and only if it is similarity [subsimilarity, monosimilarity, projectivity] factorial.

    \item If, in \ref*{e-distances:equiv}, we use two-sided associativity as the equivalence relation on atoms, we obtain the \emph{permutable distance} $\sd_p$.
      We say that $H$ is \emph{permutably factorial} if it is $\sd_p$-factorial.
      For a commutative cancellative semigroup $H$, the permutable distance is just the usual distance.
  \end{enumerate}
\end{example}

Having a rigorous notion of factorizations and distances between them at our disposal, it is now straightforward to introduce catenary degrees.

\begin{definition}[Catenary degree]
  Let $H$ be an atomic cancellative small category, $\sd$ a distance on $H$, and $a \in H$.
  \begin{enumerate}
    \item Let $z$,~$z' \in \sZ^*(a)$ and $N \in \bN_0$.
      A finite sequence of rigid factorizations $z_0$,~$\ldots\,$,~$z_n \in \sZ^*(a)$, where $n \in \bN_0$, is called an \emph{$N$-chain (in distance $\sd$)} between $z$ and $z'$ if
      \[
        z = z_0,\  z' = z_n,\ \text{ and } \ \sd(z_{i-1}, z_{i}) \le N\ \text{ for all $i \in [1,n]$.}
      \]

    \item The \emph{catenary degree (in distance $\sd$) of $a$}, denoted by $\sc_\sd(a)$, is the minimal $N \in \bN_0 \cup \{\infty\}$ such that for any two factorizations $z$,~$z' \in \sZ^*(a)$ there exists an $N$-chain between $z$ and $z'$.

    \item The \emph{catenary degree (in distance $\sd$) of $H$} is
      \[
      \sc_\sd(H) = \sup\{\, \sc_\sd(a) \mid a \in H \,\} \in \bN_0 \cup \{\infty\}.
      \]
  \end{enumerate}
\end{definition}
To abbreviate the notation, we write $\sc^*$ instead of $\sc_{\sd^*}$, $\sc_p$ instead of $\sc_{\sd_p}$, and so on.

Note that $H$ is $\sd$-factorial if and only if it is atomic and $\sc_\sd(H) = 0$.
Hence, the catenary degree provides a more fine grained arithmetical invariant than those derived from sets of lengths.

\begin{example}
  If $R$ is an atomic $2$-fir, it follows from the usual inductive proof of the Jordan-Hölder Theorem that $\sc^*(R^\bullet) \le 2$.
  Since $R$ is similarity factorial, $\sc_{\dsim}(R^\bullet) = 0$, where $\sc_{\dsim}$ denotes the catenary degree with respect to the similarity distance.
  However, $\sc^*(R^\bullet) =0$ if and only if $R$ is rigid.
  More generally, if $R$ is an atomic modular LCM domain, then $\sc^*(R^\bullet) \le 2$, and $\sc_{\text{proj}}(R^\bullet) = 0$, where the latter stands for the catenary degree in the projectivity distance.
\end{example}

The definitions of the monotone and the equal catenary degree can similarly be extended to the noncommutative setting.
For the permutable distance it is also possible to introduce an $\omega_p$-invariant $\omega_p(H)$ and a tame degree $\st_p(H)$ (see \cite[Section 5]{baeth-smertnig15}).
Unfortunately, these notions are not as strong as in the commutative setting.

\subsection{FF-Domains}

Faced with an atomic domain with non-unique factorizations, a first question one can ask is when $R$ is a BF-domain, that is, $\card{\sL(a)} < \infty$ for all $a \in R^\bullet$, respectively an FF-domain, that is, $\card{\sZ^*(a)} < \infty$ for all $a \in R^\bullet$.
A useful sufficient condition for $R$ to be a BF-domain is the existence of a length function (see Lemma~\ref{l-length-bf}).

In \cite{bell-heinle-levandovskyy}, Bell, Heinle, and Levandovskyy give a sufficient condition for many important noncommutative domains to be FF-domains.
Let $K$ be a field and $R$ a $K$-algebra.
A \emph{finite-dimensional filtration} of $R$ is a filtration of $R$ by finite-dimensional $K$-subspaces.
\begin{theorem}[{\cite[Corollary 1.2]{bell-heinle-levandovskyy}}]
  Let $K$ be a field, $\overline K$ an algebraic closure of $K$, and let $R$ be a $K$-algebra.
  If there exists a finite-dimensional filtration on $R$ such that the associated graded ring $\gr R$ has the the property that $\gr R \otimes_K \overline K$ is a domain, then $R$ is an FF-domain.
\end{theorem}
The proof of the theorem proceeds by (classical) algebraic geometry.

\begin{definition}
  Let $K$ be a field and $n \in \bN$.
  For $i,j \in [1,n]$ with $i < j$, let $c_{i,j} \in K^\times$ and $d_{i,j} \in K\langle x_1,\ldots,x_n \rangle$.
  The $K$-algebra
  \[
  R = K \langle x_1,\ldots,x_n \mid x_j x_i = c_{i,j} x_i x_j + d_{i,j},\; i < j \in [1,n] \rangle
  \]
  is called a \emph{G-algebra} (or \emph{PBW algebra}, or \emph{algebra of solvable type}) if
  \begin{enumerate}
    \item the family of monomials $\cM=(\, x_1^{k_1}\cdots x_n^{k_n})_{(k_1,\ldots,k_n) \in \bN_0^n}$ is a $K$-basis of $R$, and
    \item there exists a monomial well-ordering $\prec$ on $\cM$ such that, for all $i<j \in [1,n]$ either $d_{i,j}=0$, or the leading monomial of $d_{i,j}$ is smaller than $x_i x_j$ with respect to $\prec$.
  \end{enumerate}
\end{definition}
\begin{remark}
The family of monomials $\cM$ is naturally in bijection with $\bN_0^n$.
A monomial well-ordering on $\cM$ is a total order on $\cM$ such that, with respect to the corresponding order on $\bN_0^n$, the  semigroup $\bN_0^n$ is a totally ordered semigroup, and such that $\mathbf 0$ is the least element of $\bN_0^n$.
By Dickson's lemma, this implies that the order is a well-ordering.
\end{remark}
\begin{corollary}[{\cite[Theorem 1.3]{bell-heinle-levandovskyy}}]
  Let $K$ be a field.
  Then $G$-algebras over $K$ as well as their subalgebras are FF-domains.
  In particular, the following algebras are FF-domains:
  \begin{enumerate}
    \item Weyl algebras and shift algebras,
    \item universal enveloping algebras of finite-dimensional Lie algebras,
    \item coordinate rings of quantum affine spaces,
    \item $q$-shift algebras and $q$-Weyl algebras,
  \end{enumerate}
  as well as polynomial rings over any of these algebras.
\end{corollary}
In addition, explicit upper bounds on the number of factorizations are given in \cite[Theorem 1.4]{bell-heinle-levandovskyy}.

The following example shows that even for very nice domains (e.g., PIDs) one cannot in general expect there to be only finitely many rigid factorizations for each element.
\begin{example} \label{e-pid-not-ff}
  Let $Q$ be a quaternion division algebra over a (necessarily infinite) field $K$ with $\chr(K) \ne 2$.
  Let $a \in Q^\times \setminus K^\times$.
  We denote by $\overline{a}$ the conjugate of $a$.
  Then $\nr(a) = a\overline{a} \in K^\times$ and $\tr(a) = a + \overline{a} \in K$.
  For all $c \in Q^\times$,
  \[
    f = x^2 - \tr(a) x + \nr(a) = (x - cac^{-1})(x - c\overline{a}c^{-1}) \in Q[x],
  \]
  and thus $\card{\sZ^*(f)} = \card{Q^\times}$ is infinite.
  Hence $Q[x]$ is not an FF-domain.
  However, being Euclidean, $Q[x]$ is similarity factorial, that is, $\card{\sZ_{\dsim}(f)} = 1$ for all $f \in Q[x]^\bullet$.
\end{example}

\begin{remark}
  Another sufficient condition for a domain or semigroup to have finite rigid factorizations is given in \cite[Theorem 5.23.1]{smertnig13}.
\end{remark}

\subsection{Transfer Homomorphisms}

Transfer homomorphisms play an important role in the theory of non-unique factorizations in the commutative setting.
A transfer homomorphism allows us to express arithmetical invariants of a ring, semigroup, or small category in terms of arithmetical invariants of a possibly simpler object.

In the commutative setting, a particularly important transfer homomorphism is that from a commutative Krull monoid $H$ to the monoid of zero-sum sequences $\cB(G_0)$ over a subset $G_0$ of the class group $G$ of $H$.
In particular, if $H = \cO^\bullet$ with $\cO$ a holomorphy ring in a global field, then $G_0=G$ is a finite abelian group.
This allows one to study the arithmetic of $H$ through combinatorial and additive number theory.
(See \cite{geroldinger09}.)

In a noncommutative setting, transfer homomorphisms were first explicitly used by Baeth, Ponomarenko, Adams, Ardila, Hannasch, Kosh, McCarthy, and Rosenbaum in the article \cite{baeth-ponomarenko-etal11}.
They studied non-unique factorizations in certain subsemigroups of $M_n(\bZ)^\bullet$ and $T_n(\bZ)^\bullet$.
Transfer homomorphisms for cancellative small categories have been introduced in \cite{smertnig13}, where the main application was to classical maximal orders in central simple algebras over global fields.
This has been developed further in \cite{baeth-smertnig15}, where arithmetical invariants going beyond sets of lengths were studied.

Implicitly, the concept of a transfer homomorphism was already present in earlier work due to Estes and Matijevic (in \cite{estes-matijevic79a, estes-matijevic79b}), who essentially studied when $\det\colon M_n(R)^\bullet \to R^\bullet$ is a transfer homomorphism, and Estes and Nipp (in \cite{estes-nipp89,estes91a,estes91b}), who essentially studied when the reduced norm in a  classical hereditary $\cO$-order over a holomorphy ring $\cO$ is a transfer homomorphism.
Unfortunately, their results seem to have been largely overlooked so far.

We recall the necessary definitions.
See \cite[Section 2]{baeth-smertnig15} for more details.
\begin{definition}[Transfer homomorphism]
  Let $H$ and $T$ be cancellative small categories.
  A homomorphism $\phi\colon H \to T$ is called a \emph{transfer homomorphism} if it has the following properties:
  \begin{enumerate}[label=\textup{(\textbf{T\arabic*})},ref=\textup{(T\arabic*)}]
  \item\label{th:units} $T=T^\times \phi(H)T^{\times}$ and $\phi^{-1}(T^{\times})=H^{\times}$.
  \item\label{th:lift} If $a \in H$, $b_1$, $b_2 \in T$ and $\phi(a)=b_1b_2$, then there exist $a_1$,~$a_2 \in H$ and $\varepsilon \in T^\times$ such that $a = a_1a_2$, $\phi(a_1) = b_1 \varepsilon^{-1}$, and $\phi(a_2) = \varepsilon b_2$.
  \end{enumerate}
\end{definition}

We denote by $T_n(D)$ the ring of $n \times n$ upper triangular matrices over a commutative domain $D$.
To study $T_n(D)^\bullet$, weak transfer homomorphisms were introduced by Bachman, Baeth, and Gossell in \cite{bachman-baeth-gossell14}.

\begin{definition}[Weak transfer homomorphism]
  Let $H$ and $T$ be cancellative small categories, and suppose that $T$ is atomic.
  A homomorphism $\phi\colon H \rightarrow T$ is called a \emph{weak transfer homomorphism} if it has the following properties:
  \begin{enumerate}[label=\textup{(\textbf{WT\arabic*})},ref=\textup{(WT\arabic*)}]
  \item[\textbf{(T1)}] $T=T^\times \phi(H)T^{\times}$ and $\phi^{-1}(T^{\times})=H^{\times}$.
    \setcounter{enumi}{1}
  \item\label{wth:lift} If $a \in H$, $n \in \bN$, $v_1$, $\ldots\,$,~$v_n \in \cA(T)$ and $\phi(a)=v_1\cdots v_n$, then there exist $u_1$, $\ldots\,$,~$u_n \in \cA(H)$ and a permutation $\sigma \in \fS_n$ such that $a=u_1\cdots u_n$ and $\phi(u_i) \simeq v_{\sigma(i)}$ for each $i \in [1,n]$.
  \end{enumerate}
\end{definition}
(Weak) transfer homomorphisms map atoms to atoms.
If $a \in H$, property \ref*{th:lift} of a transfer homomorphism allows one to lift rigid factorizations of $\phi(a)$ in $T$ to rigid factorizations of $a$ in $H$.
For a weak transfer homomorphism, \ref*{wth:lift} allows the lifting of rigid factorizations of $\phi(a)$ up to permutation and associativity.
These properties are sufficient to obtain an equality of the system of sets of lengths of $H$ and $T$ (see Theorem~\ref{t-transfer} below).

To obtain results about the catenary degree, in the case where $\phi$ is a transfer homomorphism, we need additional information about the fibers of the induced homomorphism $\phi^*\colon \sZ^*(H) \to \sZ^*(T)$.

\begin{definition}[Catenary degree in the permutable fibers]
  Let $H$ and $T$ be atomic cancellative small categories, and let $\sd$ be a distance on $H$.
  Suppose that there exists a transfer homomorphism $\phi\colon H \to T$.
  Denote by $\phi^*\colon \sZ^*(H) \to \sZ^*(T)$ its natural extension to the categories of rigid factorizations.

  \begin{enumerate}
    \item
      Let $a \in H$, and let $z$,~$z' \in \sZ^*(a)$ with $\sd_p(\phi^*(z), \phi^*(z')) = 0$.
      We say that an $N$-chain $z=z_0$,~$z_1$, $\ldots\,$,~$z_{n-1}$,~$z_n=z' \in \sZ^*(a)$ of rigid factorizations of $a$ \emph{lies in the permutable fiber of $z$} if $\sd_p(\phi^*(z_i),\phi^*(z))=0$ for all $i \in [0,n]$.

    \item
      We define $\sc_\sd(a,\phi)$ to be the smallest $N \in \bN_0 \cup \{\infty\}$ such that, for any two $z$,~$z' \in \sZ^*(a)$ with $\sd_p(\phi^*(z),\phi^*(z'))=0$, there exists an $N$-chain (in distance $\sd$) between $z$ and $z'$, lying in the permutable fiber of $z$.
      Moreover, we define the \emph{catenary degree in the permutable fibers}
      \[
      \sc_\sd(H, \phi) = \sup\big\{\, \sc_\sd(a, \phi) \mid a \in H \,\big\} \in \bN_0 \cup \{\infty \}.
      \]
    \end{enumerate}
\end{definition}

For the following basic result on [weak] transfer homomorphisms, see \cite{baeth-smertnig15} and also \cite{smertnig13,bachman-baeth-gossell14}.
\begin{theorem} \label{t-transfer}
  Let $H$ and $T$ be cancellative small categories.
  Let $\phi\colon H \to T$ be a transfer homomorphism, or let $T$ be atomic and $\phi\colon H \to T$ a weak transfer homomorphism.
  \begin{enumerate}
  \item $H$ is atomic if and only if $T$ is atomic.
  \item For all $a \in H$, $\sL_H(a) = \sL_T(\phi(a))$.
    In particular $\cL(H) = \cL(T)$, and all arithmetical invariants from Definition~\ref*{d-len-inv} coincide for $H$ and $T$.
  \item If $\phi$ is a transfer homomorphism and $H$ is atomic, then
      \begin{align*}
        \sc_\sd(H) & \le \max\big\{ \sc_p(T), \sc_\sd(H, \phi) \big\}.
      \end{align*}
    \item If $\phi$ is an isoatomic weak transfer homomorphism (that is, $\phi(a) \simeq \phi(b)$ implies $a \simeq b$) and $T$ is atomic, then $\sc_p(H) = \sc_p(T)$.
      If, moreover, $T$ is an atomic commutative semigroup, then $\omega_p(H) = \omega_p(T)$ and $\st_p(H) = \st_p(T)$.
  \end{enumerate}
\end{theorem}

The strength of a transfer result comes from being able to find transfer homomorphism to a codomain $T$ which is significantly easier to study than the original category $H$.
Monoids of zero-sum sequences have played a central role in the commutative theory, and also turn out to be useful in studying classical maximal orders in central simple algebras over global fields.
We recall their definition and some of the basic structural results about their arithmetic.

Let $(G,+)$ be an additively written abelian group, and let $G_0 \subset G$ be a subset.
In the tradition of combinatorial number theory, elements of the multiplicatively written free abelian monoid $\cF(G_0)$ are called \emph{sequences over $G_0$}.
The inclusion $G_0 \subset G$ extends to a homomorphism $\sigma\colon \cF(G_0) \to G$.
Explicitly, if $S=g_1\cdot \ldots \cdot g_l \in \cF(G_0)$ is a sequence, written as a formal product of elements of $G_0$, then $\sigma(S) = g_1 + \cdots + g_l \in G$ is its sum in $G$.
We call $S$ a \emph{zero-sum sequence} if $\sigma(S)=0$.
The subsemigroup
\[
\cB(G_0) = \big\{\, S \in \cF(G_0) \mid \sigma(S) = 0 \,\big\}
\]
of the free abelian monoid $\cF(G_0)$ is called the \emph{monoid of zero-sum sequences over $G_0$}.
(See \cite{geroldinger09} or \cite[Chapter 2.5]{ghk06}.)

The semigroup $\cB(G_0)$ is a Krull monoid.
It is of particular importance in the theory of non-unique factorizations since every commutative Krull monoid [domain] $H$ possesses a transfer homomorphism to a monoid of zero-sum sequences over a subset of the class group of $H$.
Thus, problems about non-unique factorizations in $H$ can often be reduced to questions about $\cB(G_0)$.

Factorization problems in $\cB(G_0)$ are studied with methods from combinatorial and additive number theory.
Motivated by the study of rings of algebraic integers, the case where $G_0=G$ is a finite abelian group has received particular attention.
We recall some of the most important structural results in this case.
See \cite[Definition 3.2.2]{geroldinger09} for the definition of an almost arithmetical multiprogression (AAMP).

\begin{theorem} \label{t-zss}
  Let $G$ be a finite abelian group, and let $H=\cB(G)$ be the monoid of zero-sum sequences over $G$.
  \begin{enumerate}
  \item $H$ is half-factorial if and only if $\card{G} \le 2$.
  \item The set of distances, $\Delta(H)$, is a finite interval, and if it is non-empty, then $\min \Delta(H) = 1$.
  \item For every $k \in \bN$, the union of sets of lengths containing $k$, $\cU_k(H)$, is a finite interval.
  \item\label{t-zss:structure} There exists an $M \in \bN_0$ such that for every $a \in H$ the set of lengths $\sL(a)$ is an AAMP with difference $d \in \Delta (H)$ and bound $M$.
  \end{enumerate}
\end{theorem}
The last result, \ref*{t-zss:structure}, is called the \emph{Structure Theorem for Sets of Lengths}, and is a highly non-trivial result on the general structure of sets of lengths.
We give a short motivation for it.
Suppose that $H$ is a cancellative semigroup and an element $a$ has two factorizations of distinct length, say $a=u_1\cdots u_k$ and $a=v_1\ldots v_l$ with $k < l$ and atoms $u_i$, $v_j \in \cA(H)$.
That is, $\{k, l  = k + (l-k) \} \subset \sL(a)$.
Then $a^n = (u_1\cdots u_k)^i (v_1\cdots v_l)^{n-i}$ for all $i \in [0,n]$.
Hence the arithmetical progression $\{\, k + i(l-k) \mid i \in [0,n] \,\}$ with difference $l-k$ and length $n+1$ is contained in $\sL(a^n)$.
Additional pairs of lengths of $a$ give additional arithmetical progressions in $\sL(a^n)$.

If everything is ``nice,'' we might hope that this is essentially the only way that large sets of lengths appear.
Consequently, we would expect large sets of lengths to look roughly like unions of long arithmetical progressions.
The Structure Theorem for Sets of Lengths implies that this is indeed so in the setting above: If $a \in H$, then $\sL(a)$ is contained in a union of arithmetical progressions with some difference $d \in \Delta(H)$, and with possible gaps at the beginning and at the end.
The size of these gaps is uniformly bounded by the parameter $M$ which only depends on $H$ and not the particular element $a$.

\subsection{Transfer Results}
\label{ssec:transfer-results}

In this section, we gather transfer results for matrix rings, triangular matrix rings, and classical hereditary and maximal orders in central simple algebras over global fields.

\subsubsection{Matrix Rings}
\label{sssec:matrix}
For $R$ a $2n$-fir, factorizations in $M_n(R)$ have been studied by P.\,M.~Cohn.
In the special case where $R$ is a commutative PID, the existence of the Smith normal form implies that $\det\colon M_n(R)^\bullet \to R^\bullet$ is a transfer homomorphism.
This was noted in \cite{baeth-ponomarenko-etal11}.

Let $R$ be a commutative ring.
In \cite{estes-matijevic79a,estes-matijevic79b}, Estes and Matijevic studied when $M_n(R)$ has \emph{[weak] norm-induced factorization}, respectively \emph{determinant-induced factorization}.
Here, $M_n(R)$ has determinant-induced factorization if for each $A \in M_n(R)$ and each $r \in R^\bullet$ which divides $\det(A)$, there exists a right divisor of $A$ having determinant $r$.
We do not give the definition of [weak] norm-induced factorization, but recall the following.
\begin{proposition} \label{pnife}
  Let $R$ be a commutative ring and $n \in \bN$.
  Consider the following statements:
  \begin{equivenumerate}
    \item\label{pnife:nif} $M_n(R)$ has norm-induced factorization.
    \item\label{pnife:dif} $M_n(R)$ has determinant-induced factorization.
    \item\label{pnife:th} $\det\colon M_n(R)^\bullet \to R^\bullet$ is a transfer homomorphism.
  \end{equivenumerate}
  Then \ref*{pnife:nif}${}\Rightarrow{}$\ref*{pnife:dif}${}\Rightarrow{}$\ref*{pnife:th}.
  If $R$ is a finite direct product of Krull domains, then also the converse implications hold.
\end{proposition}
\begin{proof}
  The implications \ref*{pnife:nif}${}\Rightarrow{}$\ref*{pnife:dif}${}\Rightarrow{}$\ref*{pnife:th} follow immediately from the definitions and the fact that a matrix  $A \in M_n(R)$ is a zero-divisor if and only if $\det(A) \in R$ is a zero-divisor.
  Suppose that $R$ is a finite direct product of Krull domains.
  Then \ref*{pnife:dif}${}\Rightarrow{}$\ref*{pnife:nif} holds by \cite[Proposition~5]{estes-matijevic79a}, and \ref*{pnife:th}${}\Rightarrow{}$\ref*{pnife:dif} can be deduced from \cite[Lemma~2]{estes-matijevic79a}.
\end{proof}

In the characterization of rings $R$ for which $M_n(R)$ has norm-induced factorization, the notion of a \emph{Towber ring} (see \cite{towber68,lissner-geramita70}) appears.
We do not recall the exact definition, but give a sufficient as well as a necessary condition for $R$ to be Towber when $R$ is a commutative Noetherian domain.
There is a small gap between the sufficient and the necessary condition.

Let $R$ be a commutative Noetherian domain.
If $\gldim(R) \le 2$ and every finitely generated projective $R$-module is isomorphic to a direct sum of a free module and an ideal of $R$, then $R$ is a Towber ring.
Conversely, if $R$ is a Towber ring, then $\gldim(R) \le 2$ and every finitely generated projective $R$-module of rank at least 3 is isomorphic to a direct sum of a free module and an ideal.
\begin{theorem}[{\cite{estes-matijevic79a}}] \label{t-nif}
  Let $R$ be a commutative Noetherian ring with no nonzero nilpotent elements.
  Then the following statements are equivalent.
  \begin{equivenumerate}
    \item\label{t-nif:nnif} $M_n(R)$ has norm-induced factorization for all $n \in \bN$.
    \item\label{t-nif:2nif} $M_2(R)$ has norm-induced factorization.
    \item\label{t-nif:towber} $R$ is a Towber ring, that is, $R$ is a finite direct product of Towber domains.
    \item\label{t-nif:proj} $\gldim(R) \le 2$, each projective module $P$ of constant rank $r(P)$ is stably equivalent to $\bigwedge^{r(P)} P$, and stably free finitely generated projective $R$-modules are free.
  \end{equivenumerate}

  Moreover, the statements above imply the following statements \ref*{t-nif:ndif}--\ref*{t-nif:2th}.
  If $R$ is a finite direct product of Noetherian integrally closed domains, then the converse holds, and any of the above statements \ref*{t-nif:nnif}--\ref*{t-nif:proj} is equivalent to any of the statements \ref*{t-nif:ndif}--\ref*{t-nif:2th}.
  \begin{equivenumerate}
    \setcounter{enumi}{4}
    \item \label{t-nif:ndif} $M_n(R)$ has determinant-induced factorization for all $n \in \bN$.
    \item \label{t-nif:2dif} $M_2(R)$ has determinant-induced factorization.
    \item \label{t-nif:nth} $\det\colon M_n(R)^\bullet \to R^\bullet$ is a transfer homomorphism for all $n \in \bN$.
    \item \label{t-nif:2th} $\det\colon M_2(R)^\bullet \to R^\bullet$ is a transfer homomorphism.
  \end{equivenumerate}
\end{theorem}

\begin{proof}
  The equivalences \ref*{t-nif:nnif}${}\Leftrightarrow{}$\ref*{t-nif:2nif}${}\Leftrightarrow{}$\ref*{t-nif:towber}${}\Leftrightarrow{}$\ref*{t-nif:proj}, and more, follow from \cite[Theorem~1]{estes-matijevic79a}.
  The remaining implications follow from Proposition~\ref{pnife}.
\end{proof}

\begin{theorem}[{\cite{estes-matijevic79b}}] \label{thm:em-matrix-sim-fact}
  Let $R$ be a commutative Noetherian ring with no nonzero nilpotent elements.
  Then the following statements are equivalent.
  \begin{equivenumerate}
    \item\label{tsf:np} $M_n(R)$ is permutably factorial for all $n \in \bN$.
    \item\label{tsf:2p} $M_2(R)$ is permutably factorial.
    \item\label{tsf:n} $M_n(R)$ is similarity factorial for all $n \in \bN$.
    \item\label{tsf:2} $M_2(R)$ is similarity factorial.
    \item\label{tsf:pid} $R$ is a finite direct product of PIDs.
  \end{equivenumerate}
\end{theorem}

\begin{proof}
  Here, \ref*{tsf:np}${}\Rightarrow{}$\ref*{tsf:2p} and \ref*{tsf:n}${}\Rightarrow{}$\ref*{tsf:2} are trivial.
  Since associated elements are similar, \ref*{tsf:np}${}\Rightarrow{}$\ref*{tsf:n} and \ref*{tsf:2p}${}\Rightarrow{}$\ref*{tsf:2} are also clear.
  The key implication \ref*{tsf:2}${}\Rightarrow{}$\ref*{tsf:pid} follows from \cite[Theorem~2]{estes-matijevic79b}.
  Finally, \ref*{tsf:pid}${}\Rightarrow{}$\ref*{tsf:np} follows using the Smith Normal Form.
  (The implication \ref*{tsf:pid}${}\Rightarrow{}$\ref*{tsf:n} can also be deduced from Theorem~\ref{t-nfir-sf}.)
\end{proof}

In \cite{estes-matijevic79b}, the ring $M_n(R)$ is called \emph{determinant factorial} if factorizations of elements in $M_n(R)^\bullet$ are unique up to order and associativity of the determinants of the atoms.
If $M_n(R)$ is similarity factorial, then it is determinant factorial (by \cite[Proposition 5]{estes-matijevic79b}.
\begin{theorem}[{\cite{estes-matijevic79b}}] \label{thm:em-matrix-det-fact}
  Let $R$ be a commutative Noetherian ring with no nonzero nilpotent elements.
  Then the following statements are equivalent.
  \begin{equivenumerate}
    \item\label{tdf:n-df} $M_n(R)$ is determinant factorial for all $n \in \bN$.
    \item\label{tdf:2-df} $M_2(R)$ is determinant factorial.
    \item\label{tdf:n-r-f} $R^\bullet$ is factorial and for all $n \in \bN$ and $U \in \cA(M_n(R)^\bullet)$ we have $\det(U) \in \cA(R^\bullet)$.
    \item\label{tdf:2-r-f} $R^\bullet$ is factorial and for all $U \in \cA(M_2(R)^\bullet)$ we have $\det(U) \in \cA(R^\bullet)$.
    \item\label{tdf:towber} $R$ is a finite direct product of factorial Towber domains.
  \end{equivenumerate}
\end{theorem}

\begin{proof}
  The implications \ref*{tdf:n-df}${}\Rightarrow{}$\ref*{tdf:2-df} and  \ref*{tdf:n-r-f}${}\Rightarrow{}$\ref*{tdf:2-r-f} are clear.
  The equivalences \ref*{tdf:n-df}${}\Leftrightarrow{}$\ref*{tdf:n-r-f} and \ref*{tdf:2-df}${}\Leftrightarrow{}$\ref*{tdf:2-r-f} follow from \cite[Proposition 1]{estes-matijevic79b}.
  Finally, \ref*{tdf:2-df}${}\Rightarrow{}$\ref*{tdf:towber} is the key implication and follows from \cite[Theorem 1]{estes-matijevic79b}, and \ref*{tdf:towber}${}\Rightarrow{}$\ref*{tdf:n-df} follows from \cite[Corollary to Proposition~1]{estes-matijevic79b} or Theorem~\ref{t-nif}.
\end{proof}

The following example from \cite{estes-matijevic79a} forms the basis of a key step in \cite{estes-matijevic79b}.
We recall it here, as it demonstrates explicitly that a matrix ring over a factorial commutative domain need not even be half-factorial.
\begin{example}
  \begin{enumerate}
  \item
  Let $R$ be a commutative ring containing elements $x$, $y$, $z$ which form a regular sequence in any order.
  (E.g., if $R$ is a regular local ring of dimension at least $3$, three elements from a minimal generating set of the maximal ideal of $R$ will do. Also $R=K[x,y,z]$ with $K$ a field works.)

  Consider the ring $M_2(R)$.
  In \cite{estes-matijevic79a} it is shown that the matrix
  \[
  A=\begin{pmatrix} x^2 & xy - z \\ xy + z & y^2 \end{pmatrix},
  \]
  which has $\det(A)=z^2$, has no right factor of determinant $z$.
  Let $\adj(A)$ denote the adjugate of $A$.
  Then
  \[
  A\adj(A) = z^2 1_{M_2(R)} = \begin{pmatrix} z & 0 \\ 0 & 1 \end{pmatrix}^2 \begin{pmatrix} 1 & 0 \\ 0 & z \end{pmatrix}^2.
  \]
  Hence $\rho_2(M_2(R)^\bullet) \ge 4$.
  In particular, for the elasticity we have $\rho(M_2(R)^\bullet) \ge 2$.

  \item
    Let $K$ be a field.
    Then $M_2(K[x])$ is permutably, similarity, and determinant factorial.
    The ring $M_2(K[x,y])$ is determinant factorial but neither similarity nor permutably factorial.
    For $n \ge 3$, the ring $M_2(K[x_1,\ldots,x_n])$ is not even half-factorial.
  \end{enumerate}
\end{example}

\subsubsection{Rings of Triangular Matrices.}
\label{sssec:triangular}

For a commutative domain $R$ and $n \in \bN$, let $T_n(R)$ denote the ring of $n\times n$ upper triangular matrices.
The study of factorizations in $T_n(R)^\bullet$ turns out to be considerably simpler than in $M_n(R)^\bullet$.
\begin{theorem}
  Let $R$ be an atomic commutative domain and let $n \in \bN$.
  \begin{enumerate}
    \item Suppose $R$ is a BF-domain and $n \ge 2$.
      Then $\det\colon T_n(R)^\bullet \to R^\bullet$ is a transfer homomorphism if and only if $R$ is a PID.
    \item The map $T_n(R)^\bullet \to (R^\bullet)^n$ sending a matrix $(a_{i,j})_{i,j \in [1,n]} \in T_n(R)^\bullet$ to the vector of its diagonal entries $(a_{i,i})_{i\in[1,n]}$ is an isoatomic weak transfer homomorphism.

      Moreover, for atoms of $T_n(R)^\bullet$, associativity, similarity, and subsimilarity coincide, $\sc_p(T_n(R)^\bullet)=\sc_p(R^\bullet)$, $\st_p(T_n(R)^\bullet)=\st(R^\bullet)$, and $\omega_p(T_n(R)^\bullet) = \omega(R^\bullet)$.
      In particular, $T_n(R)$ is permutably [similarity, subsimilarity, determinant] factorial if and only if $R$ is factorial.
  \end{enumerate}
\end{theorem}

\begin{remark}
  \begin{enumerate}
  \item
  The existence of the transfer homomorphism, in case $R$ is a PID, was shown in \cite{baeth-ponomarenko-etal11}.
  The characterization of when the determinant is a transfer homomorphism, in case $R$ is a BF-domain, as well as the existence of a weak transfer homomorphism, is due to \cite[Theorems 2.8 and 4.2]{bachman-baeth-gossell14}.
  The isoatomicity and transfer of catenary degree, tame degree, and $\omega_p$-invariant can be found in \cite[Proposition 6.14]{baeth-smertnig15}.

  \item
  That $T_n(R)$ is determinant factorial if and only if $R^\bullet$ is factorial was not stated before, but is easy to observe.
  If $R$ is factorial, then $T_n(R)^\bullet$ is permutably factorial and hence determinant factorial.
  For the converse, suppose that $T_n(R)^\bullet$ is determinant factorial, and consider the embedding $R^\bullet \to T_n(R)^\bullet$ that maps $a \in R^\bullet$ to the matrix with $a$ in the upper left corner, ones on the remaining diagonal, and zeroes everywhere else.

  \item
  In general, there does not exist a transfer homomorphism from $T_2(R)^\bullet$ to any cancellative commutative semigroup (see \cite[Example~4.5]{bachman-baeth-gossell14}).
  This was the motivation for the introduction of weak transfer homomorphisms.
  \end{enumerate}
\end{remark}

\subsubsection{Classical Hereditary and Maximal Orders.}

Earlier results of Estes and Nipp in \cite{estes-nipp89,estes91a,estes91b} on \emph{factorizations induced by norm factorization (FNF)} can be interpreted as a transfer homomorphism.
The following is proved for central separable algebras in \cite{estes91a}.
We state the special case for central simple algebras.
\begin{theorem}\label{thm:estes-detf}
  Let $\cO$ be a holomorphy ring in a global field $K$, and let $A$ be a central simple $K$-algebra.
  Assume that $A$ satisfies the Eichler condition with respect to $\cO$.
  If $R$ is a classical hereditary $\cO$-order in $A$, $x \in R$, and $a \in \cO$ is such that $a \mid \nr(x)$, then there exists a left divisor $y$ of $x$ in $R$, and $\varepsilon \in \cO^\times$ such that $\nr(y) = a \varepsilon$.
  Moreover, $\varepsilon$ can be taken arbitrarily subject to the restriction that $a\varepsilon$ is positive at each archimedean place of $K$ which ramifies in $A$.
\end{theorem}
The proof in \cite{estes91a} proceeds by localization and an explicit characterization of classical hereditary orders over complete DVRs.
For quaternion algebras, more refined results, not requiring the Eichler condition but instead requiring that every stably free right $R$-ideal is free, can be found in \cite{estes-nipp89,estes91b}.

Let $\cO_A^\bullet$ denote the subsemigroup of all nonzero elements of $\cO$ which are positive at each archimedean place of $K$ which ramifies in $A$.
Recall that $\nr(R^\bullet)=\cO_A^\bullet$ if $R$ is a classical hereditary $\cO$-order.
\begin{corollary}\label{cor:estes-transfer}
  With the conditions as in the previous theorem, $\nr\colon R^\bullet \to \cO_A^\bullet$ is a transfer homomorphism.
  The semigroup $\cO_A^\bullet$ is a Krull monoid with class group $\Cl_A(\cO)$.
  Each class in $\Cl_A(\cO)$ contains infinitely many prime ideals.
  Hence, there exists a transfer homomorphism $R^\bullet \to \cB(\Cl_A(\cO))$.
  In particular, the conclusions of Theorem~\ref{t-zss} hold for $R^\bullet$ in place of $H$.
\end{corollary}
\begin{proof}
  By the previous theorem, $\nr \colon R^\bullet \to \cO_A^\bullet$ is a transfer homomorphism.
  The semigroup $\cO_A^\bullet$ is a regular congruence submonoid of $\cO_A^\bullet$ (see \cite[Chapter 2.11]{ghk06}).
  As such it is a commutative Krull monoid, with class group $\Cl_A(\cO)$.
  Each class contains infinitely many prime ideals by a standard result from analytic number theory.
  (See \cite[Corollary 7 to Proposition 7.9]{narkiewicz04} or \cite[Corollary 2.11.16]{ghk06} for the case where $\cO$ is the ring of algebraic integers in a number field.
  The general number field case follows by a localization argument.
  For the function field case, use \cite[Proposition~8.9.7]{ghk06}.)
  Hence there exists a transfer homomorphism $\cO_A^\bullet \to \cB(\Cl_A(\cO))$.
  Since the composition of two transfer homomorphisms is a transfer homomorphism, it follows that there exists a transfer homomorphism $R^\bullet \to \cB(\Cl_A(\cO))$.
\end{proof}
A different way of obtaining the result in Corollary~\ref{cor:estes-transfer} in the case that $R$ is a classical maximal order is given in Theorem~\ref{t-transfer-global}\ref*{t-transfer-global:transfer} below.
It relies on the global ideal theory of $R$.
In this way, we also obtain information about the catenary degree in the permutable fibers.

We first extend the result about the transfer homomorphism for commutative Krull monoids into a setting of noncommutative semigroups, respectively cancellative small categories.
This general result then includes, as a special case, the transfer homomorphism for normalizing Krull monoids obtained in \cite{geroldinger13} as well as the desired theorem.
We follow \cite{smertnig13,baeth-smertnig15}.

A \emph{quotient semigroup} is a semigroup $Q$ in which every cancellative element is invertible, that is, $Q^\bullet = Q^\times$.
Let $Q$ be a quotient semigroup and $H \subset Q$ a subsemigroup.
Then $H$ is an \emph{order} in $Q$ if $Q = H (H \cap Q^\bullet)^{-1} = (H \cap Q^\bullet)^{-1} H$.
Two orders $H$ and $H'$ in $Q$ are equivalent if there exist $x$,~$y$, $z$, $w \in Q^\bullet$ such that $xHy \subset H'$ and $zH'w \subset H$.
A \emph{maximal order} is an order which is maximal with respect to set inclusion in its equivalence class.
Let $H$ be a maximal order.
A subset $I \subset Q$ is a \emph{fractional right $H$-ideal} if $IH \subset I$, and there exist $x$,~$y \in Q^\bullet$ such that $x \in I$ and $yI \subset Q$.
If moreover $I \subset H$, then $I$ is a \emph{right $H$-ideal}.

For a fractional right $H$-ideal $I \subset Q$, we define $I^{-1} = \{\, x \in Q \mid IxI \subset I \,\}$, and $I_v = (I^{-1})^{-1}$.
The fractional right $H$-ideal $I$ is called \emph{divisorial} if $I = I_v$.
A divisorial right $H$-ideal $I$ is \emph{maximal integral} if it is maximal within the set of proper divisorial right $H$-ideals.
Analogous definitions are made for (fractional) left $H$-ideals.
If $H$ and $H'$ are equivalent maximal orders, we call a subset $I \subset Q$ a \emph{[fractional] $(H,H')$-ideal} if it is both, a [fractional] left $H$-ideal and a [fractional] right $H'$-ideal.
A \emph{[fractional] $H$-ideal} is a [fractional] $(H,H)$-ideal.
We say that $H$ is \emph{bounded} if every fractional left $H$-ideal and every fractional right $H$-ideal contains a fractional $H$-ideal.

The additional restrictions imposed in the following definition ensure that the set of maximal orders equivalent to $H$ has a ``good'' theory of divisorial left and right ideals.

\begin{definition}[{\cite[Definition 5.18]{smertnig13}}] \label{amo}
  Let $H$ be a maximal order in a quotient semigroup $Q$.
  We say that $H$ is an \emph{arithmetical maximal order} if it has the following properties:
  \begin{enumerate}[label=\textup{(\textbf{A\arabic*})},ref=(\textup{A\arabic*})]
    \item\label{amo:acc} $H$ satisfies both the ACC on divisorial left $H$-ideals and the ACC on divisorial right $H$-ideals.
    \item\label{amo:bdd} $H$ is bounded.
    \item\label{amo:mod} The lattice of divisorial fractional left $H$-ideals is modular, and the lattice of divisorial fractional right $H$-ideals is modular.
  \end{enumerate}
\end{definition}

Let $H$ be an arithmetical maximal order in a quotient semigroup $Q$, and let $\alpha$ denote the set of maximal orders in its equivalence class.
We define a category $\cF_v=\cF_v(\alpha)$ as follows: the set of objects is $\alpha$, and for $H'$, $H'' \in \alpha$, the set of morphisms from $H'$ to $H''$, denoted by $\cF_v(H',H'')$, consists of all divisorial fractional $(H',H'')$-ideals.
If $I \in \cF_v(H',H'')$ and $J \in \cF_v(H'',H''')$, the composition $I\cdot_v J \in \cF_v(H',H''')$ is defined by $I \cdot_v J = (IJ)_v$.
In terms of our point of view from the preliminaries, $\cF_v(\alpha)_0=\alpha$, and for a divisorial fractional $(H',H'')$-ideal $I$ we have that $s(I)=H'$ is the left order of $I$, and $t(I)=H''$ is the right order of $I$.

With these definitions, $\cF_v$ is an \emph{arithmetical groupoid}, the precise definition of which we omit here.
By $\cI_v = \cI_v(\alpha)$, we denote the subcategory of $\cF_v(\alpha)$ with the same set of objects, but where the morphisms are given by divisorial $(H',H'')$-ideals.
Set $\cH_H = \{\, q^{-1}(aH)q \mid a \in H^\bullet,\, q \in Q^\bullet \,\}$ (as a category).

The subcategory $\cF_v(H)$ of all divisorial fractional $H$-ideals is a free abelian group.
If $H' \in \alpha$, then there is a canonical isomorphism $\cF_v(H) \to \cF_v(H')$.
We identify, and call this group $\bG$.
One can define a homomorphism, the abstract norm, $\eta \colon G \to \bG$.
Set $P_{H^\bullet}$ to be the quotient group of $\eta(\cH_H)$ as a subgroup of $\bG$.

\begin{theorem}[{\cite[Theorem 5.23]{smertnig13} and \cite[Corollary 7.11]{baeth-smertnig15}}]
  Let $H$ be an arithmetical maximal order in a quotient semigroup $Q$ and let $\alpha$ denote the set of maximal orders of $Q$ equivalent to $H$.
  Let $\eta\colon \cF_v(\alpha) \to \bG$ be the abstract norm of $\cF_v(\alpha)$, let $C = \bG / P_{H^\bullet}$, and set $C_M = \{\, [\eta(I)] \in C \mid I \in \cI_v(\alpha) \text{ maximal integral} \,\}$.
  Assume that
  \begin{enumerate}[label=\textup{(\textbf{N})},ref=\textup{(\textbf{N})}]
  \item\label{ass:norm} a divisorial fractional right $H$-ideal $I$ is principal if and only if $\eta(I) \in P_{H^\bullet}$.
  \end{enumerate}

  Then there exists a transfer homomorphism $\theta\colon H^\bullet \to \cB(C_M)$.
  Let $\sd$ be a distance on $H^\bullet$ that is invariant under conjugation by normalizing elements.
  Then $\sc_{\sd}(H^\bullet,\theta) \le 2$.
\end{theorem}

\begin{remark}
  \begin{enumerate}
    \item
      The result can be proven in the more general setting of saturated subcategories of arithmetical groupoids (see \cite[Theorem 4.15]{smertnig13} or \cite[Theorem 7.8]{baeth-smertnig15}).
      The strong condition \ref*{ass:norm} cannot be omitted.
      We discuss the condition in our application to classical maximal orders in central simple algebras over global fields below.

    \item In a saturated subcategory of an arithmetical groupoid (here, $\cI_v$ in $\cF_v$), elements (i.e., divisorial one-sided ideals) enjoy a kind of unique factorization property.
      The boundedness guarantees the existence of the abstract norm, which provides a useful invariant in describing these factorizations.
      This was originally proven by Asano and Murata in \cite{asano-murata53}.
      It is a generalization of a similar result for (bounded) Dedekind prime rings, where the one-sided ideals of the equivalence class of a Dedekind prime ring form the so-called \emph{Brandt groupoid}.
      This unique factorization of divisorial one-sided ideals is the key ingredient in the proof of the previous result.

    \item
      We note in passing that every arithmetical maximal order is a BF-semigroup (see \cite[Theorem 5.23.1]{smertnig13}).
      For a commutative cancellative semigroup $H$ the following is true:
      If $H$ is $v$-Noetherian (satisfies the ACC on divisorial ideals), then $H$ is a BF-monoid.
      It seems to be unknown whether every order $H$ which satisfies \ref{amo:acc} is a BF-semigroup, even in the special case where $H$ is a maximal order.

    \item
      If $G$ is a lattice-ordered group, then $G$ is distributive as a lattice.
      From this, one concludes that a commutative cancellative semigroup that is a maximal order (i.e., completely integrally closed) and satisfies \ref{amo:acc} is already an arithmetical maximal order (that is, a commutative Krull monoid).

      If $H=R$ with $R$ a Dedekind prime ring, or more generally, a Krull ring in the sense of Chamarie (see \cite{chamarie81}), then \ref{amo:mod} holds.
      It is open whether there exist maximal orders which satisfy \ref{amo:acc} and \ref{amo:bdd} but not \ref{amo:mod}.
      It would be interesting to know such examples or sufficient and/or necessary conditions on $H$ for \ref{amo:acc} and \ref{amo:bdd} to imply \ref{amo:mod}.
    \end{enumerate}
\end{remark}

Applied to classical maximal orders in central simple algebras over global fields, we have the following.
(See also Corollary~\ref{cor:estes-transfer}.)
\begin{theorem}[{\cite{smertnig13,baeth-smertnig15}}] \label{t-transfer-global}
  Let $\cO$ be a holomorphy ring in a global field $K$, $A$ a central simple algebra over $K$, and $R$ a classical maximal $\cO$-order of $A$.
  \begin{enumerate}
  \item\label{t-transfer-global:transfer}
    Suppose that every stably free right $R$-ideal is free. Then there exists a transfer homomorphism $\theta \colon R^\bullet \to \cB(\Cl_A(\cO))$.
    Moreover, $\sc_{\sd}(R^\bullet,\theta) \le 2$ for any distance $\sd$ on $R^\bullet$ which is invariant under conjugation by normalizing elements.

    In particular, the conclusions of Theorem~\ref{t-zss} hold for $R^\bullet$ in place of $H$.
    If $R$ is not half-factorial, then $\sc_{\dsim}(R^\bullet) = \sc_{p}(R^\bullet) = \sc^*(R^\bullet) = \sc_p\big( \cB(\Cl_A(\cO)) \big)$.

  \item\label{t-transfer-global:non-transfer}
    Let $K$ be a number field and $\cO$ its ring of algebraic integers.
    If there exist stably free right $R$-ideals that are not free, then there exists no transfer homomorphism $\theta \colon R^\bullet \to \cB(G_0)$, where $G_0$ is any subset of an abelian group. Moreover,
    \begin{enumerate}
      \item $\Delta(R^\bullet) = \bN$.
      \item For every $k \ge 3$, we have $\bN_{\ge 3} \subset \cU_k(R^\bullet) \subset \bN_{\ge_2}$.
      \item $\sc_{\sd}(R^\bullet) = \infty$ for every distance $\sd$ on $R^\bullet$.
    \end{enumerate}
  \end{enumerate}
\end{theorem}

\begin{remark}
  \begin{enumerate}
  \item
    The importance of the condition for every stably free right $R$-ideal to be free was noted already by Estes and Nipp (see \cite{estes-nipp89,estes91b}).
    That the absence of this condition not only implies that $\nr$, respectively $\theta$, is not a transfer homomorphism, but that the much stronger result in \ref*{t-transfer-global:non-transfer} holds, first appeared in \cite{smertnig13}.
    In the setting of \ref*{t-transfer-global:non-transfer}, arithmetical invariants are infinite and hence the factorization theory is radically different from the case \ref*{t-transfer-global:transfer}, where all arithmetical invariants are finite.

  \item
    Throughout this section we have required that $\cO=\cO_S$ is a holomorphy ring defined by a finite set of places $S \subset S_{\text{fin}}$.
    This is the most important case.
    However, most results go through, with possibly minor modifications, for $\cO=\cO_S$ with an infinite set $S \subsetneq S_{\text{fin}}$.

    Theorem~\ref{thm:estes-detf} remains true without changes.
    In Corollary~\ref{cor:estes-transfer} and Theorem~\ref{t-transfer-global}\ref*{t-transfer-global:transfer} it is not necessarily true anymore that every class of $\Cl_A(\cO)$ contains infinitely many prime ideals.
    However, by a localization argument, every class, except possibly the trivial one, contains at least one nonzero prime ideal.
    Accordingly, there exists a transfer homomorphism $R^\bullet \to \cB(C_M)$ with $C_M$ either equal to $\Cl_A(\cO)$ or to $\Cl_A(\cO) \setminus \{0\}$.
  \end{enumerate}
\end{remark}

It was noted in \cite{estes91a}, that Theorem~\ref{thm:estes-detf} can be extended to a more general setting of classical hereditary orders over Dedekind domains whose quotient fields are not global fields.
In fact, using a description of finitely generated projective modules over hereditary Noetherian prime (HNP) rings, one can extend the construction of the transfer homomorphism to bounded HNP rings.
We refer to \cite{levy-robson11} for background on hereditary Noetherian prime (HNP) rings.

If $R$ is a HNP ring, one can define a class group $G(R)$.
If $R$ is a Dedekind prime ring, then simply $G(R) = \ker(\udim\colon K_0(R) \to \bZ)$.
Let $G_0 \subset G(R)$ denote the subset of classes $[I]-[R]$, where $I$ is a right $R$-ideal such that the composition series of $R/I$ consists precisely of one tower of $R$.

\begin{theorem}[{\cite{smertnig-hnp}}]
  Let $R$ be a bounded hereditary Noetherian prime ring.
  Suppose that every stably free right $R$-ideal is free.
  Then there exists a transfer homomorphism $\theta\colon R^\bullet \to \cB(G_0)$.
\end{theorem}

\section{Other Results}
\label{sec:scope}

Finally, we note some recent work which is beyond the scope of this article, but may conceivably be considered to be factorization theory.

In a noncommutative setting, even in the (similarity) factorial case, many interesting questions in describing factorizations in more detail remain.
Factorizations of (skew) polynomials over division rings have received particular attention.
This is especially true for Wedderburn polynomials (also called $W$-polynomials).
Some recent work in this direction due to Haile, Lam, Leroy, Ozturk, and Rowen is \cite{lam-leroy87,haile-rowen95,lam-leroy00,lam-leroy04,lam-leroy-ozturk08}.
In \cite{leroy12}, Leroy shows that factorizations of elements in $\bF_q[x;\theta]$, where $\theta$ is the Frobenius automorphism, can be computed in terms of factorizations in $\bF_q[x]$.
We also note \cite{heinle-levandovskyy,bergen-giesbrecht-pappur-zhang15}.

I.~Gelfand and Retakh, using their theory of quasideterminants and noncommutative symmetric functions, have obtained noncommutative generalizations of Vieta's Theorem.
This allows one to express coefficients of polynomials in terms of pseudo-roots.
We mention the surveys \cite{gelfand-gelfand-retakh-wilson05,retakh10} as starting points into the literature in this direction.
In \cite{delenclos-leroy07}, a connection is made between the theory of quasideterminants, noncommutative symmetric functions, and $W$-polynomials.

Motion polynomials are certain polynomials over the ring of dual quaternions.
They have applications in the the study of rational motions and in particular the construction of linkages in kinematics.
This approach was introduced by Hegedüs, Schicho, and Schröcker in \cite{hegedues-schicho-schroecker12,hegedues-schicho-schroecker13} and has since been very successful.
See the survey \cite{li-rad-schicho-schroecker} or also the expository article \cite{hegedues-li-schicho-schroecker15}.

We mainly discussed the semigroup of non-zero-divisors of a noncommutative ring, and, in Section~\ref{ssec:chatters-ufds}, the semigroup of nonzero normal elements.
The factorization theory of some other noncommutative semigroups, which do not necessarily arise in such a way from rings, has been studied.
We mention polynomial decompositions (see \cite{zieve-mueller}) and  other subsemigroups of rings of matrices (see \cite{baeth-ponomarenko-etal11}) over the integers.

\begin{acknowledgement}
  I thank the anonymous referee for his careful reading.
  The author was supported by the Austrian Science Fund (FWF) project P26036-N26.
\end{acknowledgement}

\bibliographystyle{alphaabbr}
\bibliography{survey}

\newcommand{\etalchar}[1]{$^{#1}$}
\begin{thebibliography}{GGRW05}

\bibitem[AKMU91]{abbasi-kobayashi-marubayashi-ueda91}
G.~Q. Abbasi, S.~Kobayashi, H.~Marubayashi, and A.~Ueda.
\newblock Noncommutative unique factorization rings.
\newblock {\em Comm. Algebra}, 19(1):167--198, 1991.

\bibitem[AM53]{asano-murata53}
K.~Asano and K.~Murata.
\newblock Arithmetical ideal theory in semigroups.
\newblock {\em J. Inst. Polytech. Osaka City Univ. Ser. A. Math.}, 4:9--33,
  1953.

\bibitem[AM16]{akalan-marubayashi}
E.~Akalan and H.~Marubayashi.
\newblock {M}ultiplicative {I}deal {T}heory in {N}on-{C}ommutative {R}ings.
\newblock In {\em {M}ultiplicative {I}deal {T}heory and {F}actorization
  {T}heory}. Springer, 2016.
\newblock To appear.

\bibitem[And97]{anderson97}
D.~D. Anderson, editor.
\newblock {\em Factorization in integral domains}, volume 189 of {\em Lecture
  Notes in Pure and Applied Mathematics}. Marcel Dekker, Inc., New York, 1997.

\bibitem[BBG14]{bachman-baeth-gossell14}
D.~Bachman, N.~R. Baeth, and J.~Gossell.
\newblock Factorizations of upper triangular matrices.
\newblock {\em Linear Algebra Appl.}, 450:138--157, 2014.

\bibitem[Bea71]{beauregard71}
R.~A. Beauregard.
\newblock Right {${\rm LCM}$} domains.
\newblock {\em Proc. Amer. Math. Soc.}, 30:1--7, 1971.

\bibitem[Bea74]{beauregard74}
R.~A. Beauregard.
\newblock Right-bounded factors in an {LCM} domain.
\newblock {\em Trans. Amer. Math. Soc.}, 200:251--266, 1974.

\bibitem[Bea77]{beauregard77}
R.~A. Beauregard.
\newblock An analog of {N}agata's theorem for modular {LCM} domains.
\newblock {\em Canad. J. Math.}, 29(2):307--314, 1977.

\bibitem[Bea80]{beauregard80}
R.~A. Beauregard.
\newblock Left versus right {LCM} domains.
\newblock {\em Proc. Amer. Math. Soc.}, 78(4):464--466, 1980.

\bibitem[Bea92]{beauregard92}
R.~A. Beauregard.
\newblock Unique factorization in the ring {$R[x]$}.
\newblock {\em J. Austral. Math. Soc. Ser. A}, 53(3):287--293, 1992.

\bibitem[Bea93]{beauregard93}
R.~A. Beauregard.
\newblock When is {$F[x,y]$} a unique factorization domain?
\newblock {\em Proc. Amer. Math. Soc.}, 117(1):67--70, 1993.

\bibitem[Bea94a]{beauregard94}
R.~A. Beauregard.
\newblock Factorization in {LCM} domains with conjugation.
\newblock {\em Canad. Math. Bull.}, 37(3):289--293, 1994.

\bibitem[Bea94b]{beauregard94b}
R.~A. Beauregard.
\newblock Right unique factorization domains.
\newblock {\em Rocky Mountain J. Math.}, 24(2):483--489, 1994.

\bibitem[Bea95]{beauregard95}
R.~A. Beauregard.
\newblock Rings with the right atomic multiple property.
\newblock {\em Comm. Algebra}, 23(3):1017--1026, 1995.

\bibitem[Ber68]{bergman68}
G.~M. Bergman.
\newblock {\em Commuting elements in free algebras and related topics in ring
  theory}.
\newblock 1968.
\newblock PhD thesis, Harvard University.

\bibitem[BGSZ15]{bergen-giesbrecht-pappur-zhang15}
J.~Bergen, M.~Giesbrecht, P.~N. Shivakumar, and Y.~Zhang.
\newblock Factorizations for difference operators.
\newblock {\em Adv. Difference Equ.}, pages 2015:57, 6, 2015.

\bibitem[BHL15]{bell-heinle-levandovskyy}
J.~P. Bell, A.~Heinle, and V.~Levandovskyy.
\newblock On {N}oncommutative {F}inite {F}actorization {D}omains.
\newblock {\em Trans. Amer. Math. Soc.}, 2015.
\newblock To appear.

\bibitem[BK00]{berrick-keating00}
A.~J. Berrick and M.~E. Keating.
\newblock {\em An introduction to rings and modules with {$K$}-theory in view},
  volume~65 of {\em Cambridge Studies in Advanced Mathematics}.
\newblock Cambridge University Press, Cambridge, 2000.

\bibitem[BPA{\etalchar{+}}11]{baeth-ponomarenko-etal11}
N.~R. Baeth, V.~Ponomarenko, D.~Adams, R.~Ardila, D.~Hannasch, A.~Kosh,
  H.~McCarthy, and R.~Rosenbaum.
\newblock Number theory of matrix semigroups.
\newblock {\em Linear Algebra Appl.}, 434(3):694--711, 2011.

\bibitem[Bro85]{brown85}
K.~A. Brown.
\newblock Height one primes of polycyclic group rings.
\newblock {\em J. London Math. Soc. (2)}, 32(3):426--438, 1985.

\bibitem[Bru69]{brungs69}
H.-H. Brungs.
\newblock Ringe mit eindeutiger {F}aktorzerlegung.
\newblock {\em J. Reine Angew. Math.}, 236:43--66, 1969.

\bibitem[BS15]{baeth-smertnig15}
N.~R. Baeth and D.~Smertnig.
\newblock Factorization theory: {F}rom commutative to noncommutative settings.
\newblock {\em J. Algebra}, 441:475--551, 2015.

\bibitem[CC91]{chatters-clark91}
A.~W. Chatters and J.~Clark.
\newblock Group rings which are unique factorisation rings.
\newblock {\em Comm. Algebra}, 19(2):585--598, 1991.

\bibitem[CGW92]{chatters-gilchrist-wilson92}
A.~W. Chatters, M.~P. Gilchrist, and D.~Wilson.
\newblock Unique factorisation rings.
\newblock {\em Proc. Edinburgh Math. Soc. (2)}, 35(2):255--269, 1992.

\bibitem[Cha81]{chamarie81}
M.~Chamarie.
\newblock Anneaux de {K}rull non commutatifs.
\newblock {\em J. Algebra}, 72(1):210--222, 1981.

\bibitem[Cha84]{chatters84}
A.~W. Chatters.
\newblock Noncommutative unique factorization domains.
\newblock {\em Math. Proc. Cambridge Philos. Soc.}, 95(1):49--54, 1984.

\bibitem[Cha95]{chatters95}
A.~W. Chatters.
\newblock Unique factorisation in {P}.{I}. group-rings.
\newblock {\em J. Austral. Math. Soc. Ser. A}, 59(2):232--243, 1995.

\bibitem[Cha05]{chapman05}
S.~T. Chapman, editor.
\newblock {\em Arithmetical properties of commutative rings and monoids},
  volume 241 of {\em Lecture Notes in Pure and Applied Mathematics}. Chapman \&
  Hall/CRC, Boca Raton, FL, 2005.

\bibitem[CJ86]{chatters-jordan86}
A.~W. Chatters and D.~A. Jordan.
\newblock Noncommutative unique factorisation rings.
\newblock {\em J. London Math. Soc. (2)}, 33(1):22--32, 1986.

\bibitem[CK15]{cohn-kumar15}
H.~Cohn and A.~Kumar.
\newblock Metacommutation of {H}urwitz primes.
\newblock {\em Proc. Amer. Math. Soc.}, 143(4):1459--1469, 2015.

\bibitem[Coh62]{cohn62}
P.~M. Cohn.
\newblock Factorization in non-commutative power series rings.
\newblock {\em Proc. Cambridge Philos. Soc.}, 58:452--464, 1962.

\bibitem[Coh63a]{cohn63}
P.~M. Cohn.
\newblock Noncommutative unique factorization domains.
\newblock {\em Trans. Amer. Math. Soc.}, 109:313--331, 1963.

\bibitem[Coh63b]{cohn63b}
P.~M. Cohn.
\newblock Rings with a weak algorithm.
\newblock {\em Trans. Amer. Math. Soc.}, 109:332--356, 1963.

\bibitem[Coh65]{cohn65}
P.~M. Cohn.
\newblock Errata to: ``{N}oncommutative unique factorization domains''.
\newblock {\em Trans. Amer. Math. Soc.}, 119:552, 1965.

\bibitem[Coh69]{cohn69a}
P.~M. Cohn.
\newblock Factorization in general rings and strictly cyclic modules.
\newblock {\em J. Reine Angew. Math.}, 239/240:185--200, 1969.

\bibitem[Coh73a]{cohn73-c}
P.~M. Cohn.
\newblock Correction to: ``{U}nique factorization domains'' ({A}mer. {M}ath.
  {M}onthly {\bf 80} (1973), 1--18).
\newblock {\em Amer. Math. Monthly}, 80:1115, 1973.

\bibitem[Coh73b]{cohn73}
P.~M. Cohn.
\newblock Unique factorization domains.
\newblock {\em Amer. Math. Monthly}, 80:1--18, 1973.

\bibitem[Coh85]{cohn85}
P.~M. Cohn.
\newblock {\em Free rings and their relations}, volume~19 of {\em London
  Mathematical Society Monographs}.
\newblock Academic Press Inc. [Harcourt Brace Jovanovich Publishers], London,
  second edition, 1985.

\bibitem[Coh06]{cohn06}
P.~M. Cohn.
\newblock {\em Free ideal rings and localization in general rings}, volume~3 of
  {\em New Mathematical Monographs}.
\newblock Cambridge University Press, Cambridge, 2006.

\bibitem[CR87]{curtis-reiner87}
C.~W. Curtis and I.~Reiner.
\newblock {\em Methods of representation theory. {V}ol. {II}}.
\newblock Pure and Applied Mathematics (New York). John Wiley \& Sons, Inc.,
  New York, 1987.
\newblock With applications to finite groups and orders, A Wiley-Interscience
  Publication.

\bibitem[CS85]{cohn-schofield85}
P.~M. Cohn and A.~H. Schofield.
\newblock Two examples of principal ideal domains.
\newblock {\em Bull. London Math. Soc.}, 17(1):25--28, 1985.

\bibitem[CS03]{conway-smith03}
J.~H. Conway and D.~A. Smith.
\newblock {\em On quaternions and octonions: their geometry, arithmetic, and
  symmetry}.
\newblock A K Peters Ltd., Natick, MA, 2003.

\bibitem[DL07]{delenclos-leroy07}
J.~Delenclos and A.~Leroy.
\newblock Noncommutative symmetric functions and {$W$}-polynomials.
\newblock {\em J. Algebra Appl.}, 6(5):815--837, 2007.

\bibitem[EM79a]{estes-matijevic79a}
D.~R. Estes and J.~R. Matijevic.
\newblock Matrix factorizations, exterior powers, and complete intersections.
\newblock {\em J. Algebra}, 58(1):117--135, 1979.

\bibitem[EM79b]{estes-matijevic79b}
D.~R. Estes and J.~R. Matijevic.
\newblock Unique factorization of matrices and {T}owber rings.
\newblock {\em J. Algebra}, 59(2):387--394, 1979.

\bibitem[EN89]{estes-nipp89}
D.~R. Estes and G.~Nipp.
\newblock Factorization in quaternion orders.
\newblock {\em J. Number Theory}, 33(2):224--236, 1989.

\bibitem[Est91a]{estes91a}
D.~R. Estes.
\newblock Factorization in hereditary orders.
\newblock {\em Linear Algebra Appl.}, 157:161--164, 1991.

\bibitem[Est91b]{estes91b}
D.~R. Estes.
\newblock Factorization in quaternion orders over number fields.
\newblock In {\em The mathematical heritage of {C}. {F}. {G}auss}, pages
  195--203. World Sci. Publ., River Edge, NJ, 1991.

\bibitem[Fit36]{fitting36}
H.~Fitting.
\newblock \"{U}ber den {Z}usammenhang zwischen dem {B}egriff der
  {G}leichartigkeit zweier {I}deale und dem \"{A}quivalenzbegriff der
  {E}lementarteilertheorie.
\newblock {\em Math. Ann.}, 112(1):572--582, 1936.

\bibitem[Ger09]{geroldinger09}
A.~Geroldinger.
\newblock Additive group theory and non-unique factorizations.
\newblock In {\em Combinatorial number theory and additive group theory}, Adv.
  Courses Math. CRM Barcelona, pages 1--86. Birkh\"auser Verlag, Basel, 2009.

\bibitem[Ger13]{geroldinger13}
A.~Geroldinger.
\newblock Non-commutative {K}rull monoids: a divisor theoretic approach and
  their arithmetic.
\newblock {\em Osaka J. Math.}, 50(2):503--539, 2013.

\bibitem[GGRW05]{gelfand-gelfand-retakh-wilson05}
I.~Gelfand, S.~Gelfand, V.~Retakh, and R.~L. Wilson.
\newblock Quasideterminants.
\newblock {\em Adv. Math.}, 193(1):56--141, 2005.

\bibitem[GHK06]{ghk06}
A.~Geroldinger and F.~Halter-Koch.
\newblock {\em Non-unique factorizations}, volume 278 of {\em Pure and Applied
  Mathematics (Boca Raton)}.
\newblock Chapman \& Hall/CRC, Boca Raton, FL, 2006.
\newblock Algebraic, combinatorial and analytic theory.

\bibitem[GN10]{graetzer-nation10}
G.~Gr{\"a}tzer and J.~B. Nation.
\newblock A new look at the {J}ordan-{H}\"older theorem for semimodular
  lattices.
\newblock {\em Algebra Universalis}, 64(3-4):309--311, 2010.

\bibitem[GS84]{gilchrist-smith84}
M.~P. Gilchrist and M.~K. Smith.
\newblock Noncommutative {UFD}s are often {PID}s.
\newblock {\em Math. Proc. Cambridge Philos. Soc.}, 95(3):417--419, 1984.

\bibitem[GW04]{goodearl-warfield04}
K.~R. Goodearl and R.~B. Warfield, Jr.
\newblock {\em An introduction to noncommutative {N}oetherian rings}, volume~61
  of {\em London Mathematical Society Student Texts}.
\newblock Cambridge University Press, Cambridge, second edition, 2004.

\bibitem[GY12]{goodearl-yakimov}
K.~R. Goodearl and M.~T. Yakimov.
\newblock From quantum {O}re extensions to quantum tori via noncommutative
  {UFD}s.
\newblock 2012.
\newblock Preprint.

\bibitem[GY14]{goodearl-yakimov14}
K.~R. Goodearl and M.~T. Yakimov.
\newblock Quantum cluster algebras and quantum nilpotent algebras.
\newblock {\em Proc. Natl. Acad. Sci. USA}, 111(27):9696--9703, 2014.

\bibitem[HL14]{heinle-levandovskyy}
A.~Heinle and V.~Levandovskyy.
\newblock Factorization of $\mathbb z$-homogeneous polynomials in the {F}irst
  ($q$)-{W}eyl {A}lgebra.
\newblock 2014.
\newblock Preprint.

\bibitem[HLSS15]{hegedues-li-schicho-schroecker15}
G.~Hegedüs, Z.~Li, J.~Schicho, and H.-P. Schröcker.
\newblock {F}rom the {F}undamental {T}heorem of {A}lgebra to {K}empe’s
  {U}niversality {T}heorem.
\newblock {\em Internat. Math. Nachr.}, 229:13--26, 2015.

\bibitem[HM06]{hallouin-maire06}
E.~Hallouin and C.~Maire.
\newblock Cancellation in totally definite quaternion algebras.
\newblock {\em J. Reine Angew. Math.}, 595:189--213, 2006.

\bibitem[HR95]{haile-rowen95}
D.~Haile and L.~H. Rowen.
\newblock Factorizations of polynomials over division algebras.
\newblock {\em Algebra Colloq.}, 2(2):145--156, 1995.

\bibitem[HSS12]{hegedues-schicho-schroecker12}
G.~Hegedüs, J.~Schicho, and H.-P. Schröcker.
\newblock Construction of overconstrained linkages by factorization of rational
  motions.
\newblock In J.~Lenarcic and M.~Husty, editors, {\em Latest Advances in Robot
  Kinematics}, pages 213--220. Springer Netherlands, 2012.

\bibitem[HSS13]{hegedues-schicho-schroecker13}
G.~Hegedüs, J.~Schicho, and H.-P. Schröcker.
\newblock Factorization of rational curves in the study quadric.
\newblock {\em Mech. Machine Theory}, 69(1):142--152, 2013.

\bibitem[Jac43]{jacobson43}
N.~Jacobson.
\newblock {\em The {T}heory of {R}ings}.
\newblock American Mathematical Society Mathematical Surveys, vol. I. American
  Mathematical Society, New York, 1943.

\bibitem[JO07]{jespers-okninski07}
E.~Jespers and J.~Okni{\'n}ski.
\newblock {\em Noetherian semigroup algebras}, volume~7 of {\em Algebras and
  Applications}.
\newblock Springer, Dordrecht, 2007.

\bibitem[Jor89]{jordan89}
D.~A. Jordan.
\newblock Unique factorisation of normal elements in noncommutative rings.
\newblock {\em Glasgow Math. J.}, 31(1):103--113, 1989.

\bibitem[JW01]{jespers-wang01}
E.~Jespers and Q.~Wang.
\newblock Noetherian unique factorization semigroup algebras.
\newblock {\em Comm. Algebra}, 29(12):5701--5715, 2001.

\bibitem[JW02]{jespers-wang02}
E.~Jespers and Q.~Wang.
\newblock Height-one prime ideals in semigroup algebras satisfying a polynomial
  identity.
\newblock {\em J. Algebra}, 248(1):118--131, 2002.

\bibitem[Lam06]{lam06}
T.~Y. Lam.
\newblock {\em Serre's problem on projective modules}.
\newblock Springer Monographs in Mathematics. Springer-Verlag, Berlin, 2006.

\bibitem[Lan02]{landau02}
E.~Landau.
\newblock Ein {S}atz \"uber die {Z}erlegung homogener linearer
  {D}ifferentialausdr\"ucke in irreductible {F}actoren.
\newblock {\em J. Reine Angew. Math.}, 124:115--120, 1902.

\bibitem[LB86]{lebruyn86}
L.~Le~Bruyn.
\newblock Trace rings of generic matrices are unique factorization domains.
\newblock {\em Glasgow Math. J.}, 28(1):11--13, 1986.

\bibitem[Ler12]{leroy12}
A.~Leroy.
\newblock Noncommutative polynomial maps.
\newblock {\em J. Algebra Appl.}, 11(4):1250076, 16, 2012.

\bibitem[LG70]{lissner-geramita70}
D.~Lissner and A.~Geramita.
\newblock Towber rings.
\newblock {\em J. Algebra}, 15:13--40, 1970.

\bibitem[LL88]{lam-leroy87}
T.~Y. Lam and A.~Leroy.
\newblock Algebraic conjugacy classes and skew polynomial rings.
\newblock In {\em Perspectives in ring theory ({A}ntwerp, 1987)}, volume 233 of
  {\em NATO Adv. Sci. Inst. Ser. C Math. Phys. Sci.}, pages 153--203. Kluwer
  Acad. Publ., Dordrecht, 1988.

\bibitem[LL00]{lam-leroy00}
T.~Y. Lam and A.~Leroy.
\newblock Principal one-sided ideals in {O}re polynomial rings.
\newblock In {\em Algebra and its applications ({A}thens, {OH}, 1999)}, volume
  259 of {\em Contemp. Math.}, pages 333--352. Amer. Math. Soc., Providence,
  RI, 2000.

\bibitem[LL04]{lam-leroy04}
T.~Y. Lam and A.~Leroy.
\newblock Wedderburn polynomials over division rings. {I}.
\newblock {\em J. Pure Appl. Algebra}, 186(1):43--76, 2004.

\bibitem[LLO08]{lam-leroy-ozturk08}
T.~Y. Lam, A.~Leroy, and A.~Ozturk.
\newblock Wedderburn polynomials over division rings. {II}.
\newblock In {\em Noncommutative rings, group rings, diagram algebras and their
  applications}, volume 456 of {\em Contemp. Math.}, pages 73--98. Amer. Math.
  Soc., Providence, RI, 2008.

\bibitem[LLR06]{launois-lenagan-rigal06}
S.~Launois, T.~H. Lenagan, and L.~Rigal.
\newblock Quantum unique factorisation domains.
\newblock {\em J. London Math. Soc. (2)}, 74(2):321--340, 2006.

\bibitem[LO04]{leroy-ozturk04}
A.~Leroy and A.~Ozturk.
\newblock Algebraic and {$F$}-independent sets in 2-firs.
\newblock {\em Comm. Algebra}, 32(5):1763--1792, 2004.

\bibitem[Loe03]{loewy03}
A.~Loewy.
\newblock \"{U}ber reduzible lineare homogene {D}ifferentialgleichungen.
\newblock {\em Math. Ann.}, 56(4):549--584, 1903.

\bibitem[LR11]{levy-robson11}
L.~S. Levy and J.~C. Robson.
\newblock {\em Hereditary {N}oetherian prime rings and idealizers}, volume 174
  of {\em Mathematical Surveys and Monographs}.
\newblock American Mathematical Society, Providence, RI, 2011.

\bibitem[LRSS15]{li-rad-schicho-schroecker}
Z.~Li, T.-D. Rad, J.~Schicho, and H.-P. Schröcker.
\newblock Factorization of {R}ational {M}otions: {A} {S}urvey with {E}xamples
  and {A}pplications.
\newblock In {\em Proceedings of the 14th IFToMM World Congress (Taipei,
  2015)}, 2015.
\newblock Preprint.

\bibitem[Mar10]{marubayashi10}
H.~Marubayashi.
\newblock Ore extensions over total valuation rings.
\newblock {\em Algebr. Represent. Theory}, 13(5):607--622, 2010.

\bibitem[MR01]{mcconnell-robson01}
J.~C. McConnell and J.~C. Robson.
\newblock {\em Noncommutative {N}oetherian rings}, volume~30 of {\em Graduate
  Studies in Mathematics}.
\newblock American Mathematical Society, Providence, RI, revised edition, 2001.
\newblock With the cooperation of L. W. Small.

\bibitem[MR03]{maclachlan-reid03}
C.~Maclachlan and A.~W. Reid.
\newblock {\em The arithmetic of hyperbolic 3-manifolds}, volume 219 of {\em
  Graduate Texts in Mathematics}.
\newblock Springer-Verlag, New York, 2003.

\bibitem[MVO12]{marubayashi-vanoystaeyen12}
H.~Marubayashi and F.~Van~Oystaeyen.
\newblock {\em Prime divisors and noncommutative valuation theory}, volume 2059
  of {\em Lecture Notes in Mathematics}.
\newblock Springer, Heidelberg, 2012.

\bibitem[Nar04]{narkiewicz04}
W.~Narkiewicz.
\newblock {\em Elementary and analytic theory of algebraic numbers}.
\newblock Springer Monographs in Mathematics. Springer-Verlag, Berlin, third
  edition, 2004.

\bibitem[Ore33]{ore33}
O.~Ore.
\newblock Theory of non-commutative polynomials.
\newblock {\em Ann. of Math. (2)}, 34(3):480--508, 1933.

\bibitem[Rei75]{reiner75}
I.~Reiner.
\newblock {\em Maximal orders}.
\newblock Academic Press [A subsidiary of Harcourt Brace Jovanovich,
  Publishers], London-New York, 1975.
\newblock London Mathematical Society Monographs, No. 5.

\bibitem[Ret10]{retakh10}
V.~Retakh.
\newblock From factorizations of noncommutative polynomials to combinatorial
  topology.
\newblock {\em Cent. Eur. J. Math.}, 8(2):235--243, 2010.

\bibitem[Sme13]{smertnig13}
D.~Smertnig.
\newblock Sets of lengths in maximal orders in central simple algebras.
\newblock {\em J. Algebra}, 390:1--43, 2013.

\bibitem[Sme15]{smertnig-canc}
D.~Smertnig.
\newblock A note on cancellation in totally definite quaternion algebras.
\newblock {\em J. Reine Angew. Math.}, 707:209--216, 2015.

\bibitem[Sme16]{smertnig-hnp}
D.~Smertnig.
\newblock Factorizations in bounded hereditary noetherian prime rings.
\newblock 2016.
\newblock Preprint.

\bibitem[Swa80]{swan80}
R.~G. Swan.
\newblock Strong approximation and locally free modules.
\newblock In {\em Ring theory and algebra, {III} ({P}roc. {T}hird {C}onf.,
  {U}niv. {O}klahoma, {N}orman, {O}kla., 1979)}, volume~55 of {\em Lecture
  Notes in Pure and Appl. Math.}, pages 153--223. Dekker, New York, 1980.

\bibitem[Tow68]{towber68}
J.~Towber.
\newblock Complete reducibility in exterior algebras over free modules.
\newblock {\em J. Algebra}, 10:299--309, 1968.

\bibitem[Vig76]{vigneras76}
M.-F. Vign{\'e}ras.
\newblock Simplification pour les ordres des corps de quaternions totalement
  d\'efinis.
\newblock {\em J. Reine Angew. Math.}, 286/287:257--277, 1976.

\bibitem[ZM08]{zieve-mueller}
M.~Zieve and P.~M\"uller.
\newblock On {R}itt's polynomial decomposition theorems.
\newblock 2008.
\newblock Preprint.

\end{thebibliography}

\end{document}